\documentclass[10pt]{amsart}
\newcommand{\draftorfinal}{final}
\usepackage{mystyle}

\title{SDEs with no strong solution arising from a problem of stochastic control}
\author{Alexander M.\ G.\ Cox}
\address{\scriptsize{Department of Mathematical Sciences, University of Bath, Bath, U.K.}}
\email{\vspace{-0.5ex}\scriptsize{a.m.g.cox@bath.ac.uk}}
\author{Benjamin A.\ Robinson}
\address{\scriptsize{Universit\"at Wien, Vienna, Austria}}
\email{\vspace{-0.5ex}\scriptsize{ben.robinson@univie.ac.at}}
\date{\today}
\thanks{BR is supported by a scholarship from the EPSRC Centre for Doctoral Training in Statistical Applied Mathematics at Bath (SAMBa), under the project EP/L015684/1, and by the Austrian Science Fund (FWF) projects Y782-N25 and P35519.\\
The authors are grateful to the anonymous referee, whose comments led to strengthening the main results of this article.
}

\begin{document}

\maketitle
	
\begin{abstract}
	We study a two-dimensional stochastic differential equation that has a unique weak solution but no strong solution. We show that this SDE shares notable properties with Tsirelson's example of a one-dimensional SDE with no strong solution. In contrast to Tsirelson's equation, which has a non-Markovian drift, we consider a strong Markov martingale with Markovian diffusion coefficient. We show that there is no strong solution of the SDE and that the natural filtration of the weak solution is generated by a Brownian motion. We also discuss an application of our results to a stochastic control problem for martingales with fixed quadratic variation in a radially symmetric environment.
\end{abstract}

\section{Introduction}\label{sec:intro}

In this paper, we study the following two-dimensional SDE with Markovian diffusion coefficient, started from the origin. Let $B$ be a real-valued Brownian motion and consider the SDE
	\begin{equation}\label{eq:SDE-intro-1}
		\D X_t = \frac{1}{\abs{X_t}}
		\begin{bmatrix}
			- X^2_t\\
			X^1_t
		\end{bmatrix}
		\D B_t; \quad X_0 \sim \delta_0,
	\end{equation}
	where we denote $X_t = (X^1_t, X^2_t)^\top \in \RR^2$. It is shown by Larsson and Ruf in \cite{larsson_relative_2021} that the SDE \eqref{eq:SDE-intro-1} has a weak solution. Here we show that there does not exist a  strong solution of \eqref{eq:SDE-intro-1}. Moreover, we show that uniqueness in law holds for \eqref{eq:SDE-intro-1} and that the weak solution shares notable properties with Tsirelson's example of an SDE with no strong solution given in \cite{tsirelson_example_1976}. In particular, the natural filtration of the weak solution is generated by a Brownian motion, which implies that the initial sigma-algebra is trivial. We also show that the angle process of the solution is independent of its increments and deduce that it is independent of the driving Brownian motion. Together, these properties imply that the filtration generated by the weak solution at any positive time contains some additional information not present at time zero. This remarkable property of Tsirelson's equation is emphasised by Rogers and Williams in \cite[V.18]{rogers_diffusions_2000}.

Tsirelson's example is a one-dimensional SDE with path-dependent drift. A result of Zvonkin \cite{zvonkin_transformation_1974} shows that this path dependence is necessary; for a one-dimensional SDE of the form $\D X_t = b_t(X_t) \D t + \D W_t$, with $b$ bounded and measurable, a strong solution always exists. In contrast to Tsirelson's example, the SDE \eqref{eq:SDE-intro-1} that we study defines a two-dimensional martingale, and the diffusion coefficient is Markovian.

We will show that the weak solution of \eqref{eq:SDE-intro-1} generates a Brownian filtration, by making use of a connection with circular Brownian motion. We take inspiration from the paper \cite{azema_remark_1999}, in which \'Emery and Schachermayer showed that a weak solution of Tsirelson's equation generates a Brownian filtration, by constructing a bijection with a circular Brownian motion.

In \cite{cox_optimal_2021}, we studied a control problem for martingales with a fixed quadratic variation, for which we can explicitly identify the value function and the optimal controls. Under certain conditions, the weak solution of \eqref{eq:SDE-intro-1} is optimal. However, under a particular growth condition on the cost function, the question of whether weak and strong versions of the control problem coincide is left open. In this paper, we will show that such problems are in fact equivalent, and that the cost induced by the weak solution of \eqref{eq:SDE-intro-1} attains the strong value function.

\subsection{Main results}
	The main contribution of this paper is to present an SDE for a martingale with Markovian diffusion coefficient, which has no strong solution and shares many interesting properties with Tsirelson's path-dependent one-dimensional example from \cite{tsirelson_example_1976}.
	\begin{theorem}\label{thm:intro-1}
		There exists a unique (in law) weak solution $((X, W), \psp, \mathbb{F})$ of the SDE \eqref{eq:SDE-intro-1}, but there is no strong solution. Moreover,
		\begin{itemize}
			\item the process $X$ generates a Brownian filtration; 
			\item after a deterministic time-change, the angle process of $X$ is uniformly distributed and independent of the driving Brownian motion;
			\item taking the supremum of the natural filtration of $B$ at any time $t > 0$ and the sigma-algebra generated by the angle process of $X$ at any time $s \in (0, t)$ recovers the natural filtration of $X$ at time $t$;
			\item the process $X$ is a strong Markov process.
		\end{itemize}
	\end{theorem}
	
	We further obtain a result on non-existence of strong solutions of the following SDEs, whose behaviour approximates that of solutions of the SDE \eqref{eq:SDE-intro-1}. Let $B$ be a one-dimensional Brownian motion and let $\lambda \in (0,1)$ be a fixed constant. Consider the two-dimensional SDE
		\begin{equation}\label{eq:SDE-intro-2}
			\D X_t = \frac{1}{\abs{X_t}}
			\begin{bmatrix}
				\lambda X^1_t - \sqrt{1 - \lambda^2} X^2_t\\
				\lambda X^2_t + \sqrt{1 - \lambda^2} X^1_t
			\end{bmatrix}
			\D B_t;
			\quad X_0 \sim \delta_0.
		\end{equation}
	\begin{theorem}\label{thm:intro-2}
		There exists no strong solution of the SDE \eqref{eq:SDE-intro-2}. Uniqueness in law holds for \eqref{eq:SDE-intro-2} up to the first hitting time of the origin.
	\end{theorem}
	We will also show that, after a deterministic time change, the radius of the weak solution of \eqref{eq:SDE-intro-2} is a $\lambda^{-2}$-dimensional Bessel process.

\subsection{SDEs with no strong solution in the literature}

We begin by recalling the properties of two classical examples of SDEs with no strong solution, which will be instructive for the study of the SDE \eqref{eq:SDE-intro-1}. We emphasise the significance of Tsirelson's example in \Cref{sec:tsirelson}. For further instructive examples, see \cite[Section 1.3]{cherny_singular_2005}.

\subsubsection{Tanaka's example}
	
	A well-known example of an SDE with no strong solution is Tanaka's SDE, which is the following one-dimensional equation:
	\begin{equation}\label{eq:tanaka}
		\D X_t = \sign(X_t) \D W_t.
	\end{equation}
	The SDE \eqref{eq:tanaka} admits a unique (in law) weak solution but no strong solution. The proof of this can be found, for example, in Example 3.5 of \cite[Chapter 5]{karatzas_brownian_1998}.
	
	To prove that there is no strong solution, the key idea is to show, using the It\^o-Tanaka formula, that the inclusion
	\begin{equation}
		\mathcal{F}^{W}_t \subseteq \mathcal{F}^{\abs{X}}_t
	\end{equation}
	holds for all $t > 0$. Then it is impossible for $X$ to be adapted to $\mathcal{F}^W$, since $\mathcal{F}^{\abs{X}}_t \subsetneq \mathcal{F}^{X}_t$ for all $t >0$.
	
	In order to prove \Cref{thm:intro-1}, we show similar inclusions to the ones above, where the increments of the solution of the SDE \eqref{eq:SDE-intro-1} play the role of the absolute value of the solution of Tanaka's SDE.
	
	\subsubsection{Tsirelson's example}\label{sec:tsirelson}
	
	In his 1976 paper \cite{tsirelson_example_1976}, Tsirelson introduced the following notable example of an SDE with no strong solution, but for which weak existence and uniqueness in law holds. Tsirelson's example is the one-dimensional equation
	\begin{equation}\label{eq:tsirelson}
		\D X_t = b(t, X_.) \D t + \D W_t,
	\end{equation}
	with initial condition $X_0 = 0$, where $b$ is chosen as follows.
	
	Fix a decreasing sequence $(t_n)_{n \in - \NN \cup \{0\}}$ such that $t_0 = 1$ and $\lim_{n \to - \infty}t_n = 0$. Denote the increments of $X$ and $t$ by $\Delta X_j = X_{t_{j}} - X_{t_{j - 1}}$ and $\Delta t_j = t_j - t_{j - 1}$, respectively, and define
	\begin{equation}\label{eq:tsirelson-drift}
		b(t, X_.) \coloneqq  \sum_{k \in -\NN} \left(\frac{\Delta X_k}{\Delta t_k} - \left\lfloor\frac{\Delta X_k}{\Delta t_k}\right\rfloor\right)\mathsf{1}_{(t_{k}, t_{k + 1}]}(t).
	\end{equation}
	At time $t \in (t_{k}, t_{k + 1}]$, for some $k \in -\NN$, $b(t, X_.)$ is the fractional part of $\frac{\Delta X_{t_k}}{\Delta t_k}$.
	
	The weak solution of the SDE \eqref{eq:tsirelson} has the following properties, as proved, for example, in Theorem 18.3 of \cite[Chapter V]{rogers_diffusions_2000}:
	\begin{enumerate}[label = (T\arabic*)]
		\item \label{it:T1} At any time $t > 0$, the natural filtration of the solution $X$ has the decomposition
		\begin{equation}
			\mathcal{F}^X_t = \mathcal{F}^B_t \vee \sigma(b(t, X_.));
		\end{equation}
		\item \label{it:T2} For each $k \in -\NN$, $b(t_k, X_.)$ is uniformly distributed on $[0, 1)$ and independent of $\mathcal{F}^B_\infty$;
		\item \label{it:T3} The sigma-algebra $\mathcal{F}^X_{0+}$ is trivial.
	\end{enumerate}
	
	Note that the drift term \eqref{eq:tsirelson-drift} in Tsirelson's SDE depends on the history of the process $X$. As remarked in \cite[Chapter V]{rogers_diffusions_2000}, for bounded drifts depending only on the current value of the process, Zvonkin proved in \cite{zvonkin_transformation_1974} that a strong solution of \eqref{eq:tsirelson} always exists. Therefore the path-dependence of the drift $b$ is necessary for strong existence to fail. We emphasise that, in contrast to Tsirelson's SDE \eqref{eq:tsirelson}, the two-dimensional SDE \eqref{eq:SDE-intro-1} defines a martingale with Markovian diffusion coefficient. Nevertheless, we will show that \eqref{eq:SDE-intro-1} exhibits similar properties to \Cref{it:T1,it:T2,it:T3} above.

\subsection{Brownian filtrations and circular Brownian motion}\label{sec:brownian-filtrations-CBM}
	A natural question that arises when considering continuous-time stochastic processes is whether the natural filtration of a process is generated by a Brownian motion. In Proposition 2 of \cite{azema_remark_1999}, \'Emery and Schachermayer define a Brownian filtration as follows.

	\begin{definition}[Brownian filtration]
		A filtration is called \emph{Brownian} if it is the natural filtration of a real-valued Brownian motion starting from the origin.
	\end{definition}
	
	Note that this definition agrees with the definition of a \emph{strong Brownian filtration} given in Mansuy and Yor's book \cite[Definition 6.1]{mansuy_random_2006}.
	
	In the case of Tanaka's SDE \eqref{eq:tanaka}, any weak solution is a Brownian motion, as discussed in Example 3.5 of \cite[Chapter 5]{karatzas_brownian_1998}, and so in this case a weak solution trivially generates a Brownian filtration. For Tsirelson's example, this question remained open until the work of \'Emery and Schachermayer in 1999 \cite{azema_remark_1999}, in which they showed that the solution of Tsirelson's equation does indeed generate a Brownian filtration.
	
	In \cite{dubins_decreasing_1996}, Dubins, Feldman, Smorodinsky and Tsirelson settled an open question by presenting an example of a process that does not generate a Brownian filtration. Their proof relies on the concept of standardness, an invariant of filtrations first introduced by Vershik in the setting of ergodic theory in his doctoral thesis \cite{vershik_approximation_1973}.
	
	Another example of a process that does not generate a Brownian filtration is the diffusion that Walsh defined in \cite{walsh_diffusion_1978}, now known as Walsh's Brownian motion. In \cite{tsirelson_triple_1997}, Tsirelson proved that Walsh's Brownian motion does not generate a Brownian filtration, by introducing a new invariant of filtrations known as cosiness. Warren later used the same technique in \cite{azema_joining_1999} to prove that sticky Brownian motion also does not generate a Brownian filtration. In \cite{azema_vershiks_2001}, \'Emery and Schachermayer provide a discussion of the relationship between the two invariants standardness and cosiness, along with further references to examples of their application.

	In order to prove that the solution of Tsirelson's equation generates a Brownian filtration, neither of these invariants are used. Rather, in \cite{azema_remark_1999}, \'Emery and Schachermayer show that there is an isomorphism between the solution of Tsirelson's equation and an eternal Brownian motion on the circle, which they call \emph{circular Brownian motion} and define as follows.
	
	\begin{definition}[circular Brownian motion]\label[definition]{def:cbm}
		Let $(\phi_t)_{t \in \RR}$ be a continuous $\RR/2\pi\ZZ$-valued process. For any $s, t \in \RR$ with $s \leq t$, denote by $\int_s^t \D \phi_r$ the $\RR$-valued random variable that depends continuously on $t$, vanishes for $t = s$, and satisfies
		\begin{equation}
			\int_s^t \D \phi_r \equiv \phi_t - \phi_s \mod 2\pi.
		\end{equation}
		Let $\FF = (\mathcal{F}_t)_{t \in \RR}$ be a filtration. We say that $\phi$ is a \emph{circular Brownian motion} for $\FF$ if $\phi$ is adapted to $\FF$ and, for each $s \in \RR$, the process
		\begin{equation}
			[s, \infty) \ni t \mapsto \int_s^t \D \phi_r
		\end{equation}
		is a standard Brownian motion for the filtration $(\mathcal{F}_t)_{t \in [s, \infty)}$.
	\end{definition}
	
	Proposition 3 of \cite{azema_remark_1999} shows that any deterministic time-change of a circular Brownian motion generates a Brownian filtration. The proof uses the notion of chopped Brownian motion and a coupling argument.
	
	In this work, we show that the angle process of the weak solution of the SDE \eqref{eq:SDE-intro-1} is a deterministic time-change of a circular Brownian motion, thus relating this SDE to Tsirelson's example. We frequently make use of the connection to circular Brownian motion and the results of \cite{azema_remark_1999} to show that the SDE \eqref{eq:SDE-intro-1} has no strong solution and that the weak solution generates a Brownian filtration. 
	
	Having shown that the weak solution $X$ of \eqref{eq:SDE-intro-1} generates a Brownian filtration, an immediate consequence will be that the initial sigma-algebra $\bigcap_{s} \sigma(X_s)$ is trivial. Therefore, for any fixed $t > 0$, $\mathcal{F}^B_t = \mathcal{F}^B_t \vee \left(\bigcap_{s \leq t} \sigma(X_s)\right)$. On the other hand, for any $s \leq t$, we will show that $\mathcal{F}^B_t \vee \sigma(X_s) = \mathcal{F}^X_t$. Since $X$ is not a strong solution, $\mathcal{F}^X_t \not \subseteq \mathcal{F}^B_t$. In fact, we have the strict inclusion
	\begin{equation}
		\mathcal{F}^B_t \vee \left(\bigcap_{s \leq t} \sigma(X_s)\right) \subsetneq \bigcap_{s \leq t} \left(\mathcal{F}^B_t \vee \sigma(X_s)\right).
	\end{equation}
	Exchanging the order of taking intersections and suprema of sigma-algebras are discussed in detail by von Weizs\"acker in \cite{von_weizsacker_exchanging_1983}. The inclusion above holds in general, and von Weizs\"acker gives conditions under which there is equality. Both the SDE \eqref{eq:SDE-intro-1} that we study in this paper and Tsirelson's example \eqref{eq:tsirelson} give continuous-time counterexamples, for which von Weizs\"acker's conditions are not satisfied and the inclusion is strict. A related discrete-time counterexample is given in \cite{von_weizsacker_exchanging_1983}.

\subsection{Application to a control problem}
	In the paper \cite{cox_optimal_2021}, we study the following control problem. We seek the value
	\begin{equation}
		\inf \EE \left[\int_0^\tau f(X_t) \D t\right],
	\end{equation}
	where the infimum is taken over a set of martingales with fixed quadratic variation, stopped on exiting a ball in $\RR^d$, and the value function $f$ is radially symmetric. The main result of \cite{cox_optimal_2021} is that there is a closed form expression for the value function, and that an optimal control is to switch between two regimes. The first of these regimes is a one-dimensional Brownian motion on a radial line, while the second is a weak solution of the SDE \eqref{eq:SDE-intro-1}. In \cite{cox_optimal_2021} we call the behaviour of solutions of \eqref{eq:SDE-intro-1} \emph{tangential motion}, since such a process moves on a tangent to its current position. The remarkable property of this process is that it is a two-dimensional martingale whose radius is deterministically increasing. In particular, for a cost function that is radially decreasing, we show that tangential motion is optimal. In general, we identify optimal controls only in a weak sense, but we show that weak and strong formulations of the control problem coincide, similarly to the results of \cite{el_karoui_capacities_2013-1}. However, when weak solutions of \eqref{eq:SDE-intro-1} are optimal, under a particular growth condition on the cost function, the question of equality between weak and strong value functions is left open in \cite{cox_optimal_2021}. In \Cref{sec:control} of the present paper we settle this question, showing that the value functions are in fact equal. Since the weak solution of the SDE \eqref{eq:SDE-intro-1} generates a Brownian filtration, we can argue by isomorphism that there is a strong control that attains the same value as the optimal weak control.

\subsection{Organisation of the article}

In \Cref{sec:sde-no-strong-solution}, we state and prove our main result \Cref{thm:no-strong-solution} on solutions of the SDE \eqref{eq:SDE-intro-1}. We start by introducing circular Brownian motion and discussing its properties in  \Cref{sec:properties-weak}. The relation between circular Brownian motion and the SDE \eqref{eq:SDE-intro-1} then leads us to show that the weak solution of \eqref{eq:SDE-intro-1} generates a Brownian filtration, among other notable properties. In \Cref{sec:non-existence}, we conclude that there exists no strong solution of \eqref{eq:SDE-intro-1}.

\Cref{sec:approx-sde} treats the class of SDEs of the form \eqref{eq:SDE-intro-2}. We show that \eqref{eq:SDE-intro-2} has similar properties to \eqref{eq:SDE-intro-1} and, in \Cref{thm:approx-no-strong}, we prove that \eqref{eq:SDE-intro-2} has no strong solution.

In \Cref{sec:control}, we apply \Cref{thm:no-strong-solution} to the control problem studied in \cite{cox_optimal_2021}. In \Cref{sec:weak-equals-strong}, we extend the main result of \cite{cox_optimal_2021}, by proving \Cref{thm:strong-extension}. In \Cref{sec:feedback}, we discuss an open question on optimality of feedback controls.

Throughout the paper, all filtrations are assumed to satisfy the usual conditions, and the following notation will be used. For a stochastic process $X$, its natural filtration augmented to satisfy the usual conditions is denoted $\FF^X = (\mathcal{F}^X_t)_t$. The quadratic variation of a process $X$ is denoted $\langle X \rangle$. The sigma-algebra generated by a random variable $\xi$ is denoted $\sigma(\xi)$.

\section{An SDE with no strong solution}\label{sec:sde-no-strong-solution}

Let $B$ be a real-valued Brownian motion and consider the two-dimensional SDE
\begin{equation}\label{eq:SDE-gamma}
		\D X_t = \frac{1}{\abs{X_t}}
		\begin{bmatrix}
			- X^2_t\\
			X^1_t
		\end{bmatrix}
		\D B_t; \quad X_0 \sim \delta_0.
\end{equation}
By \cite{larsson_relative_2021}, the SDE \eqref{eq:SDE-gamma} has a weak solution. A simulation of such a solution is shown in \Cref{fig:tangential-motion}. The main result of the present paper is that the SDE \eqref{eq:SDE-gamma} has no strong solution and, furthermore, that a weak solution exhibits many of the same properties as Tsirelson's famous example \eqref{eq:tsirelson} from \cite{tsirelson_example_1976}, including uniqueness in law. It is notable that our two-dimensional example is a strong Markov martingale with Markovian diffusion coefficient, in contrast to Tsirelson's one-dimensional SDE which has a non-Markovian drift.

\begin{figure}[h]\centering
	\includegraphics[width = 0.8\textwidth]{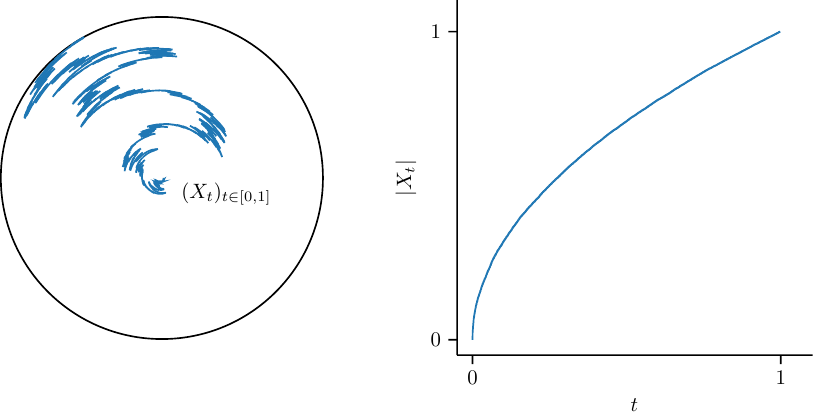}
	\caption{A simulation of a solution of the SDE \eqref{eq:SDE-gamma} (left) and its radius (right), up to the first exit time of a ball. These simulations already appeared in our paper \cite{cox_optimal_2021}.}\label{fig:tangential-motion}
\end{figure}

We now state the main result concerning solutions of the SDE \eqref{eq:SDE-gamma}. This is a more precise restatement of \Cref{thm:intro-1}

\begin{theorem}\label{thm:no-strong-solution}
	There exists a weak solution but no strong solution of the SDE \eqref{eq:SDE-gamma}. Moreover, uniqueness in law holds and the weak solution of \eqref{eq:SDE-gamma} has the following properties:
	\begin{enumerate}[label = (S\arabic*)]
		\item \label{it:S1} The natural filtration $\mathbb{F}^X$ is generated by a Brownian motion; in particular, the initial sigma-algebra $\mathcal{F}^X_{0+}$ is trivial;
		\item \label{it:S2} For any $s\in \RR$, the value of the time-changed angle process $\theta_{e^s}$ is uniformly distributed on $[0, 2\pi)$ and independent of $\mathcal{H}^\theta_\infty \coloneqq  \sigma\left(\{\theta_{e^u} - \theta_{e^r} \; \colon - \infty < r \leq u < \infty \}\right)$;
		\item \label{it:S3} For any $t > 0$, the natural filtration of $X$ at time $t$ can be decomposed as $\mathcal{F}^X_t = \mathcal{F}^B_t \; \vee \; \sigma(\theta_s)$, for any $s \in (0, t)$;
		\item \label{it:S4} The process $(X_t)_{t \geq 0}$ is a strong Markov process.
	\end{enumerate}
\end{theorem}

The existence of a weak solution of \eqref{eq:SDE-gamma} is proved by Larsson and Ruf in Theorem 4.3 of \cite{larsson_relative_2021}. We will now investigate the properties of such a weak solution and conclude that there exists no strong solution.

\subsection{Properties of weak solutions}\label{sec:properties-weak}

The key observation in our proof of \Cref{thm:no-strong-solution} is that the angle process of any solution of the SDE \eqref{eq:SDE-gamma} is a deterministic time-change of a circular Brownian motion, as defined in \Cref{def:cbm}.

We now state two properties of circular Brownian motion that are proved in \cite{azema_remark_1999}. For a circular Brownian motion $\phi$, define the \emph{innovation filtration} $\mathcal{H}$ to be the filtration generated by the increments of $\phi$;
i.e.
\begin{equation}
	\mathcal{H}_t \coloneqq  \sigma\left(\left\{\phi_s - \phi_r \colon -\infty < r \leq s \leq t\right\}\right), \quad t \in \RR.
\end{equation}
Then Proposition 1 of \cite{azema_remark_1999} states that, for any $t \in \RR$,
\begin{enumerate}[label = (C\arabic*)]
	\item \label{it:C1} $\phi_t$ is uniformly distributed;
	\item \label{it:C2} $\phi_t$ is independent of $\mathcal{H}_\infty$.
\end{enumerate}
We note the parallel between properties \Cref{it:C1,it:C2} of circular Brownian motion and the property (T2) of Tsirelson's equation \eqref{eq:tsirelson} stated in \Cref{sec:tsirelson}.

Next, we show how a circular Brownian motion arises in our example.

\begin{lemma}\label[lemma]{lem:radius-angle}
	There exists a weak solution $\left((X, B), \psp, \FF\right)$ of the SDE \eqref{eq:SDE-gamma}, and any such solution satisfies
	\begin{equation}\label{eq:radius-angle}
		X_t = \sqrt{t}
		\begin{bmatrix}
			\cos \theta_t\\
			\sin \theta_t
		\end{bmatrix},
		\quad \text{for all} \quad t > 0,
	\end{equation}
	where $\theta$ is a $\RR/2 \pi \ZZ$-valued process satisfying
	\begin{equation}\label{eq:angle-sde}
		\D \theta_t = t^{-\frac{1}{2}}\D B_t, \quad t > 0.
	\end{equation}
	In particular, the radius of $X$ is given by the deterministically increasing function $|X_t| = \sqrt t$, for $t > 0$.
\end{lemma}

\begin{proof}
	The existence of a weak solution $\left((X, B), \psp, \FF\right)$ of the SDE \eqref{eq:SDE-gamma} is given by \cite[Theorem 4.3]{larsson_relative_2021}. By It\^o's formula, for $t > 0$, we find that $|X_t| = \sqrt t$ (c.f.\ \cite[Lemma 3.1, Lemma 3.4]{cox_optimal_2021}). Now, for $t > 0$, we can write the $\RR^2$-valued random variable $X_t$ as
	\begin{equation}
		X_t = |X_t|\begin{bmatrix}
			\cos \theta_t\\
			\sin \theta_t
		\end{bmatrix}
		= \sqrt t \begin{bmatrix}
			\cos \theta_t\\
			\sin \theta_t
		\end{bmatrix},
	\end{equation}
	for a $\RR/2 \pi \ZZ$-valued random variable $\theta_t$. Applying It\^o's formula once again, we see that the process $(\theta_t)_{t > 0}$ satisfies \eqref{eq:angle-sde}.
\end{proof}

We call the process $\theta$ given in \Cref{lem:radius-angle} the \emph{angle process} of the solution $X$. We now show that this angle process is a circular Brownian motion, up to a time-change. We define a regular time-change as in \cite{azema_remark_1999}.

\begin{definition}
	A function $a: \RR \to (0, \infty)$ is a \emph{regular time-change} if $a$ is an increasing absolutely continuous bijection with absolutely continuous inverse.
\end{definition}

\begin{proposition}\label[proposition]{prop:angle-cbm}
	Let $\left((X, B), \psp, \FF\right)$ be a weak solution of the SDE \eqref{eq:SDE-gamma}. Then the associated angle process $(\theta_t)_{t > 0}$ is a regular time-change of a circular Brownian motion. Moreover, for any $s \in \RR$, the angle process is distributed as $\theta_{e^s} \sim \unif[0, 2\pi)$, independently of $\sigma\left(\left\{\theta_{e^u} - \theta_{e^r} \; \colon - \infty < r \leq u < \infty\right\}\right)$.
\end{proposition}

\begin{proof}
	Define the function $a: \RR \to (0, \infty)$ by $a(t) = e^t$, $t \in \RR$. Then $a$ is a regular time-change. Define the time-changed process
	\begin{equation}
		(\tilde{\theta}_t)_{t \in \RR} = (\theta_{a(t)})_{t \in \RR}.
	\end{equation}
	Since, for any $t > 0$, there is a one-to-one deterministic correspondence between $X_t \in \RR^2$ and $\theta_t \in \RR / 2 \pi \ZZ$, the angle process $\theta$ is adapted to $\FF$. Now define the time-changed filtration
	\begin{equation}
		\tilde{\FF} = (\tilde{\mathcal{F}}_t)_{t \in \RR} = \left(\mathcal{F}_{a(t)}\right)_{t \in \RR}.
	\end{equation}
	We will show that $\tilde{\theta}$ is a circular Brownian motion for $\tilde{\FF}$.
	
	Since $a$ is a regular time-change, $\tilde{\theta}$ is adapted to $\tilde{\FF}$. We also see that the $\RR / 2 \pi \ZZ$-valued process $\tilde{\theta}$ is continuous. Now fix $s \in \RR$ and consider the process
	\begin{equation}
		[s, \infty) \ni t \mapsto \int_s^t \D \tilde{\theta}_r = \int_s^t a(r)^{- \frac{1}{2}} \D B_{a(r)},
	\end{equation}
	using the expression \eqref{eq:angle-sde} from \Cref{lem:radius-angle}.
	
	Since $B$ is an $\FF$-Brownian motion and $a$ is a regular time-change, we have that
	\begin{equation}
		[s, \infty) \ni t \mapsto \int_s^t \D B_{a(r)}
	\end{equation}
	is a $(\tilde{\mathcal{F}}_t)_{t \in [s, \infty)}$-martingale, with quadratic variation
	\begin{equation}
		\left \langle \int_s^\cdot \D B_{a(r)}\right \rangle_t = a(t) - a(s),
	\end{equation}
	and so
	\begin{equation}
		[s, \infty) \ni t \mapsto \int_s^t \D \tilde{\theta}_r
	\end{equation}
	is a continuous $(\tilde{\mathcal{F}}_t)_{t \in [s, \infty)}$-martingale. We can calculate the quadratic variation	
	\begin{equation}
		\left\langle \int_s^\cdot \D \tilde{\theta}_r\right\rangle_t = \int_s^t a(r)^{-1} \D a(r) = t - s,
	\end{equation}
	since $a(r) = e^r$, for any $r \in \RR$.
	
	Therefore, by L\'evy's characterisation of Brownian motion, the process
	\begin{equation}
		[s, \infty) \ni t \mapsto \int_s^t \D \tilde{\theta}_r
	\end{equation}
	is an $(\tilde{\mathcal{F}}_t)_{t \in [s, \infty)}$-Brownian motion. Hence $\tilde{\theta}$ is a circular Brownian motion for $\tilde{\FF}$. It follows from properties of circular Brownian motion proved in Proposition 1 of \cite{azema_remark_1999} that, for any $s \in \RR$, $\tilde{\theta}_s$ is independent of $\sigma\left(\left\{\tilde{\theta}_u - \tilde{\theta}_r \; \colon -\infty < r \leq u < \infty\right\}\right)$ and uniformly distributed on $[0, 2\pi]$.
\end{proof}

\begin{corollary}\label[corollary]{cor:uniqueness}
	Uniqueness in law holds for the SDE \eqref{eq:SDE-gamma}.
\end{corollary}

\begin{proof}
	Let $((X, B), \psp, \FF))$ and $(\tilde X, \bar B), (\bar \Omega, \bar{\mathcal F}, \bar \PP), \bar \FF)$ be weak solutions of the SDE \eqref{eq:SDE-gamma} and write $\theta$, $\bar \theta$ for the angle processes of $X$, $\bar X$, respectively, given in \Cref{lem:radius-angle}. By \Cref{prop:angle-cbm}, each angle process is a regular time-change of a circular Brownian motion. As remarked in \cite{azema_remark_1999}, any circular Brownian motion has the same law, as a consequence of the uniformity and independence properties shown in \cite[Proposition 1]{azema_remark_1999}. Hence $\Law((\theta_t)_{t > 0}) = \Law((\bar \theta_t)_{t > 0})$. By \Cref{lem:radius-angle}, $X_t$ (resp.\ $\bar X_t$) is a deterministic function of $\theta_t$ (resp. $\bar \theta_t$), for $t > 0$, and so it follows that $\Law((X_t)_{t > 0}) = \Law((\bar X_t)_{t > 0})$.
\end{proof}

A key contribution of \'Emery and Schachermayer's paper \cite{azema_remark_1999} is to show that solutions of Tsirelson's equation generate a Brownian filtration. As discussed in \Cref{sec:brownian-filtrations-CBM}, this is done as follows. In Proposition 4 of \cite{azema_remark_1999}, the authors show that there is an isomorphism between solutions of Tsirelson's equation and circular Brownian motion, and in Proposition 3 of \cite{azema_remark_1999}, they prove that a regular time-change of a circular Brownian motion generates a Brownian filtration. It is this latter property that we exploit here, having already shown a connection between solutions of \eqref{eq:SDE-gamma} and circular Brownian motion in \Cref{prop:angle-cbm}.

\begin{corollary}\label[corollary]{cor:brownian-filtration}
	Let $\left((X, B), \psp, \FF)\right)$ be a weak solution of the SDE \eqref{eq:SDE-gamma}. Then $X$ generates a Brownian filtration.
\end{corollary}

\begin{proof}
	Write
	\begin{equation}
		X_t = \sqrt{t}
		\begin{bmatrix}
			\cos \theta_t\\
			\sin \theta_t
		\end{bmatrix},	
	\end{equation}
	where $\theta$ is the angle process of the solution, and let $\FF^X = (\mathcal{F}^X_t)_{t \geq 0}$ be the filtration generated by $X$. Then, since $X_0 = 0$ is fixed, and $X_t$ is a deterministic bijective function of $\theta_t$ for each $t > 0$, we have
	\begin{equation}
		\mathcal{F}^X_t = \mathcal{F}^\theta_t \quad \text{for all} \quad t \geq 0,
	\end{equation}
	where $\FF^\theta = (\mathcal{F}^\theta_t)_{t \geq 0}$ is the filtration generated by $\theta$.
	
	We have seen in \Cref{prop:angle-cbm} that $(\theta_t)_{t > 0}$ is a regular time-change of a circular Brownian motion. Propositions 2 and 3 of \cite{azema_remark_1999} together immediately imply that the natural filtration of any regular time-change of a circular Brownian motion is Brownian.
	Hence $\FF^\theta$ is Brownian, and it follows that $\FF^X$ is Brownian.
\end{proof}

In the next section, we will show that the SDE \eqref{eq:SDE-gamma} has no strong solution, and so the Brownian motion that generates the natural filtration of a weak solution cannot be the driving Brownian motion of the SDE.

\subsection{Non-existence of strong solutions}\label{sec:non-existence}
The proof of non-existence of a strong solution in \Cref{thm:no-strong-solution} relies on the following property of the angle process that arises from the theory of circular Brownian motion discussed in \Cref{sec:properties-weak}.

\begin{lemma}\label[lemma]{lem:cbm-not-adapted}
	Let $W$ be a real-valued Brownian motion with natural filtration $(\mathcal{F}^W_t)_{t \geq 0}$ and let $\phi$ be an $\RR/2\pi\ZZ$-valued process. Suppose that $\phi$ satisfies
	\begin{equation}
		\int_s^t \D \phi_r = \int_s^t r^{-\frac{1}{2}} \D W_r	, \quad \text{for all} \quad 0 < s \leq t,
	\end{equation}
	where the random variables on the left-hand side are defined analogously to those in \Cref{def:cbm}.
	
	Then $\phi$ cannot be adapted to $(\mathcal{F}^W_t)_{t \geq 0}$.
\end{lemma}
\begin{proof}
	Suppose for contradiction that $\phi$ is adapted to the natural filtration of $W$.
	
	Define the regular time-change $a: \RR \to (0, \infty)$ by $a(t) = e^t$ for all $t \in \RR$, as in the proof of \Cref{prop:angle-cbm}, and denote the time-changed processes
	\begin{equation}
	\begin{split}
		(\tilde{\phi}_t)_{t \in \RR} & = (\phi_{a(t)})_{t > 0},\\
		(\tilde{W}_t)_{t \in \RR} & = (W_{a(t)})_{t > 0}.
	\end{split}
	\end{equation}
	Since the time-change is deterministic, the natural filtrations $(\tilde{\mathcal{F}}^\phi_t)_{t \in \RR}$ and $(\tilde{\mathcal{F}}^W_t)_{t \in \RR}$ of the time-changed processes $\tilde{\phi}$ and $\tilde{W}$ are given by
	\begin{equation}
		\tilde{\mathcal{F}}^\phi_t = \mathcal{F}^\phi_{a(t)}, \quad \tilde{\mathcal{F}}^W_t = \mathcal{F}^W_{a(t)}, \quad \text{for all} \quad t \in \RR.
	\end{equation}
	Hence $\tilde{\phi}$ is adapted to $(\tilde{\mathcal{F}}^W_t)_{t \in \RR}$.
	
	By the same arugments as in the proof of \Cref{prop:angle-cbm}, $\tilde{\phi}$ is a circular Brownian motion for $(\tilde{\mathcal{F}}^W_t)_{t \in \RR}$ and, for any $s, t \in \RR$ with $s \leq t$,
	\begin{equation}\label{eq:cbm-integral}
		\int_s^t \D \tilde{\phi}_r = \int_s^t a(r)^{-\frac{1}{2}} \D \tilde{W}_r.
	\end{equation}
			
	To arrive at a contradiction, we will exploit a property of circular Brownian motion that is proved in Proposition 1 of \cite{azema_remark_1999}.
	
	Let $(\mathcal{H}_t)_{t \in \RR}$ be the innovation filtration of $\tilde{\phi}$. Recall that, for each $t \in \RR$, $\mathcal{H}_t$ is the sigma-algebra generated by the increments of $\tilde{\phi}$ up to time $t$; i.e.
	\begin{equation}
		\mathcal{H}_t \coloneqq  \sigma \left(\left\{\tilde{\phi}_s - \tilde{\phi}_r \colon -\infty < r \leq s \leq t\right\}\right).
	\end{equation}
	Then we have
	\begin{equation}
		\mathcal{H}_t \subseteq \tilde{\mathcal{F}}^\phi_t \subseteq \tilde{\mathcal{F}}^W_t, \quad t \in \RR.
	\end{equation}
	
	In fact, the first inclusion must be strict, as we now show. As remarked in \Cref{sec:properties-weak}, Proposition 1 of \cite{azema_remark_1999} tells us that, for each $t \in \RR$, the value of the circular Brownian motion $\tilde{\phi}_t$ is uniformly distributed on $[0, 2\pi)$ and, moreover, $\tilde{\phi}_t$ is independent of $\mathcal{H}_\infty$. Hence, for each $t \in \RR$,
	\begin{equation}\label{eq:strict-inclusion}
		\mathcal{H}_t \subsetneq \tilde{\mathcal{F}}^\phi_t \subseteq \tilde{\mathcal{F}}^W_t.
	\end{equation}
			
	Fix $0 < s \leq t$. Then, using the relation \eqref{eq:cbm-integral}, we can deduce that the increment
	\begin{equation}
		\tilde{W}_t - \tilde{W}_s = \int_s^t a(r)^\frac{1}{2} \D \tilde{\phi}_r
	\end{equation}
	is $\mathcal{H}_t$-measurable.
	
	Now, taking the limit as $s \to - \infty$, $\tilde{W}_s = W_{e^s} \to 0$ almost surely, and so $\tilde{W}_t$ is $\mathcal{H}_t$-measurable. This implies that
	\begin{equation}
		\tilde{\mathcal{F}}^W_t \subseteq \mathcal{H}_t,
	\end{equation}
	contradicting the strict inclusion in \eqref{eq:strict-inclusion}.
\end{proof}

We are now ready to prove \Cref{thm:no-strong-solution}, showing in particular that the SDE \eqref{eq:SDE-gamma} has no strong solution.

\begin{proof}[Proof of \Cref{thm:no-strong-solution}]
	As noted after the statement of the theorem, the existence of a weak solution is proved by Larsson and Ruf in Theorem 4.3 of \cite{larsson_relative_2021}. We prove uniqueness in law in \Cref{cor:uniqueness}. By \Cref{cor:brownian-filtration}, a weak solution generates a Brownian filtration, and by Blumenthal's zero-one law, the initial sigma-algebra is trivial. The uniform distribution of the time-changed angle process and its independence from its increments are shown in \Cref{prop:angle-cbm}. It remains to prove that statements \Cref{it:S3,it:S4} of the theorem hold, and that there does not exist a strong solution. We now check statement \Cref{it:S3} of the theorem.
	
	Let $((X, B), \psp, \FF))$ be a weak solution of \eqref{eq:SDE-gamma}, write $\theta$ for the angle process, and recall that $\mathbb{F}^X = \mathbb{F}^\theta$. We are in the setting of \Cref{lem:cbm-not-adapted} and so, similarly to the proof of the lemma, we find that $\tilde{\mathcal{F}}^B_t \subseteq \tilde{\mathcal{F}}^\theta_t$. Clearly, for any $s \leq t$, $\sigma(\tilde{\theta}_s) \subseteq \tilde{\mathcal{F}}^\theta_t$, and so we also have the inclusion $\tilde{\mathcal{F}}^B_t \vee \sigma(\tilde{\theta}_s) \subseteq \tilde{\mathcal{F}}^\theta_t$. On the other hand, writing
	\begin{equation}
		\tilde{\theta}_t = \tilde{\theta}_s + \int_s^t e^{- \frac{r}{2}} \D \tilde{B}_r,
	\end{equation}
	we see that $\tilde{\mathcal{F}}^\theta_t \subseteq \tilde{\mathcal{F}}^B_t \vee \sigma(\tilde{\theta}_s)$, and we conclude that $\tilde{\mathcal{F}}^X_t = \tilde{\mathcal{F}}^B_t \vee \sigma(\tilde{\theta}_s)$.
	
	To verify the strong Markov property, statement \Cref{it:S4} of the theorem, we first observe that the Markov property at time zero together with the strong Markov property on $[\varepsilon, \infty)$ for all $\varepsilon > 0$ implies the strong Markov property on $[0, \infty)$ (see e.g.\ \cite[Lemma A.2]{PaRoSc22}). Fix $\varepsilon > 0$. By \Cref{lem:radius-angle}, the radius of the weak solution is given by the deterministically increasing function $|X_t| = \sqrt t$, and so we have $|X_t| \geq \sqrt \varepsilon$ for $t \in [\varepsilon, \infty)$. The diffusion coefficient $x = (x_1, x_2)^\top \mapsto |x|^{-1}(-x_2, x_1)^\top$ of the SDE \eqref{eq:SDE-gamma} is Lipschitz on the set $\{x \in \RR^2 \colon \; |x| \geq \sqrt \varepsilon\}$, and so we can follow standard arguments  (c.f.\ \cite[Theorem 8.3]{le_gall_brownian_2016}) to show that the weak solution is in fact strong on $[\varepsilon, \infty)$. This implies that the strong Markov property holds on $[\varepsilon, \infty)$, by \cite[Corollary 8.8]{le_gall_brownian_2016}. Moreover, the Markov property at time zero follows immediately from the fact that the initial sigma-algebra is trivial. We conclude that $(X_t)_{t \geq 0}$ is a strong Markov process.
	
	To conclude the proof of \Cref{thm:no-strong-solution} it remains to show non-existence of strong solutions.
	Suppose for contradiction that $X$ is a strong solution of the SDE \eqref{eq:SDE-gamma}. Then $X$ is adapted to the filtration $(\mathcal{F}^B_t)_{t \geq 0}$; i.e
	\begin{equation}\label{eq:X-adapted}
		\mathcal{F}^X_t \subseteq \mathcal{F}^B_t, \quad t \geq 0.
	\end{equation}
	Then, since the angle process $\theta$ satisfies \eqref{eq:angle-sde}, we have
	\begin{equation}
		\int_s^t \D\theta_r = \int_s^t r^{-\frac{1}{2}} \D B_r,
	\end{equation}
	for any $0 < s \leq t$. Therefore, by \Cref{lem:cbm-not-adapted}, $\theta$ is not adapted to $(\mathcal{F}^B_t)_{t \geq 0}$.
	We have already seen in the proof of \Cref{cor:brownian-filtration} that
	\begin{equation}
		\mathcal{F}^\theta_t = \mathcal{F}^X_t, \quad \text{for all} \quad t \geq 0.
	\end{equation}
	Therefore $X$ is not adapted to $(\mathcal{F}^B_t)_{t \geq 0}$. This contradicts the inclusion \eqref{eq:X-adapted}. Hence the SDE \eqref{eq:SDE-gamma} has no strong solution.
\end{proof}
	
\begin{remark}
	Recalling \Cref{cor:brownian-filtration}, we have shown that, although the SDE \eqref{eq:SDE-gamma} has no strong solution, there exists a unique (in law) weak solution, which generates a Brownian filtration. As discussed in \Cref{sec:brownian-filtrations-CBM}, this places our example into the more common class of SDEs whose weak solutions are not strong but do generate a Brownian filtration, as is the case for the examples of Tanaka and Tsirelson.
\end{remark}

\section{Approximating SDEs have no strong solution}\label{sec:approx-sde}

In this section, we consider a class of SDEs whose behaviour approximates that of the SDE \eqref{eq:SDE-gamma} studied in \Cref{sec:sde-no-strong-solution}. We show that such SDEs exhibit similar properites to the SDE \eqref{eq:SDE-gamma} and that there do not exist strong solutions. In \Cref{sec:control} we will relate these SDEs and the SDE \eqref{eq:SDE-gamma} to a control problem for two-dimensional martingales that is studied in \cite{cox_optimal_2021}. In particular, as remarked after \Cref{prop:lambda-approx-cost}, the non-existence of strong solutions gives an insight into the problem of optimising over feedback controls ---  see \Cref{sec:feedback}. The main result of this section is the following restatement of \Cref{thm:intro-2}.

\begin{theorem}\label{thm:approx-no-strong}
	Let $B$ be a one-dimensional Brownian motion and let $\lambda \in (0,1)$ be a fixed constant. Then there exists no strong solution of the SDE
	\begin{equation}\label{eq:SDE-approx}
		\D X_t = \frac{1}{\abs{X_t}}
		\begin{bmatrix}
			\lambda X^1_t - \sqrt{1 - \lambda^2} X^2_t\\
			\lambda X^2_t + \sqrt{1 - \lambda^2} X^1_t
		\end{bmatrix}
		\D B_t;
		\quad X_0 \sim \delta_0.
	\end{equation}
	Uniqueness in law holds for \eqref{eq:SDE-approx} on the time interval $[0, \tau^\lambda_0)$, where $\tau^\lambda_0 \coloneqq \inf\{t > 0 : X_t = 0 \}$.
\end{theorem}

Note that setting $\lambda = 0$ in \eqref{eq:SDE-approx} reduces the SDE to \eqref{eq:SDE-gamma}, and so we exclude this case here.

We first observe that the squared radius process of a solution of \eqref{eq:SDE-approx} can be rescaled to a squared Bessel process, as defined in Definition 1.1 of \cite[Chapter XI]{revuz_continuous_1999}. We will show that the event of returning to the origin before leaving the domain satisfies the following zero-one law. For $\lambda \leq \frac{\sqrt{2}}{2}$, $X^\lambda$ returns to the origin with probability zero; for $\lambda > \frac{\sqrt{2}}{2}$, $X^\lambda$ returns to the origin with probability one. The critical value $\lambda = \frac{\sqrt{2}}{2}$ corresponds to the $2$-dimensional squared Bessel process, which has the same law as the squared radius process of a $2$-dimensional Brownian motion. This remark plays an important role in the study of uniqueness of multi-dimensional martingales with given marginals in \cite{PaRoSc22}.
	
\begin{proposition}\label[proposition]{prop:lambda-return-origin}
	Let $\lambda \in (0, 1)$ and suppose that $X^\lambda$ solves the SDE \eqref{eq:SDE-approx}. Write $Z^\lambda_t = \abs{X^\lambda_t}^2$ for any $t \geq 0$ and define the rescaled process $\tilde{Z}^\lambda$ by $\tilde{Z}^\lambda_t = Z_{\lambda^{-2}t}$.
	
	Then $\tilde{Z}^\lambda$ is the square of a $\delta$-dimensional Bessel process started from $0$, where $\delta = \lambda^{-2}$.
	Moreover, defining $\tau^\lambda_0 \coloneqq  \inf\{t > 0 
	\colon Z^\lambda_t = 0\}$, we have
	\begin{equation}
		\PP^0\left[\tau^\lambda_0 < \infty\right] =
		\begin{cases}
			0, & \lambda \in (0, \frac{\sqrt{2}}{2}],\\
			1, & \lambda \in (\frac{\sqrt{2}}{2}, 1).
		\end{cases}
	\end{equation}
\end{proposition}

\begin{proof}
	Applying It\^o's formula, we see that $Z^\lambda$ satisfies
	\begin{equation}
		\D Z^\lambda_t = 2 \lambda \sqrt{Z^\lambda_t} \D B_t + \D t,
	\end{equation}
	with $Z^\lambda_0 = 0$. Note that
	\begin{equation}
		t \mapsto \tilde{B}_t \coloneqq  \lambda B_{\lambda^{-2}t}
	\end{equation}
	is a standard Brownian motion. Therefore, for any $t \geq 0$,
	\begin{equation}
		\tilde{Z}^\lambda_t = 2 \int_0^t \sqrt{\tilde{Z}^\lambda_s}\D \tilde{B}_s + \lambda^{-2}t.
	\end{equation}
	Set $\delta = \lambda^{-2}$. Then, referring to Definition 1.1 of \cite[Chapter XI]{revuz_continuous_1999}, we see that $\tilde{Z}^\lambda$ is the square of a $\delta$-dimensional Bessel process.
	
	Now suppose that $\lambda \in (0, \frac{\sqrt{2}}{{2}}]$, so that
	\begin{equation}
		\delta = \lambda^{-2} \geq 2.
	\end{equation}
	The discussion that immediately precedes Proposition 1.5 in \cite[Chapter XI]{revuz_continuous_1999} tells us that the set $\{0\}$ is \emph{polar} for $\tilde{Z}^\lambda$. By the definition of a polar set given in Definition 2.6 of \cite[Chapter V]{revuz_continuous_1999}, we have that $\tilde{Z}^\lambda$ almost surely never returns to the origin in finite time, and the rescaled process $Z^\lambda$ has the same property.
	
	On the other hand, suppose that $\lambda \in (\frac{\sqrt{2}}{2}, 1)$. Then
	\begin{equation}
		\delta = \lambda^{-2} < 2,
	\end{equation}
	and so, by the same discussion in \cite[Chapter XI]{revuz_continuous_1999}, $\tilde{Z}^\lambda$ returns to the origin in finite time with probability $1$. Again the rescaled process $Z^\lambda$ has the same property.
\end{proof}

\begin{remark}\label[remark]{rem:lambda-unique-strong}
	Define the process $R^\lambda$ by $R^\lambda_t = \abs{X_t^\lambda}$, for $t \geq 0$.
	Since $\tilde Z^\lambda$ is the square of a $\lambda^{-2}$-dimensional Bessel process, we have that $t \mapsto \sqrt{\tilde Z^\lambda_t}$ is a Bessel process (see \cite[Definition XI.1.9]{revuz_continuous_1999}). Rescaling the SDE for the Bessel process (see \cite[Eq.\ (4)]{cherny_strong_2000}), we see that $R^\lambda$ satisfies
	\begin{equation}\label{eq:R-lambda-sde}
		\D R^\lambda_t = \lambda \D B_t + \frac{1 - \lambda^2}{2 R^\lambda_t} \mathsf{1}_{\{R^\lambda_t \neq 0\}} \D t; \quad R^\lambda_0 = 0.
	\end{equation}
	By \cite[Theorem 3.2 (i)]{cherny_strong_2000}, $R^\lambda$ is the unique non-negative solution of \eqref{eq:R-lambda-sde} and it is a strong solution.
	
	Suppose that $\lambda \in (0, \frac{\sqrt{2}}{2}]$. By \Cref{prop:lambda-return-origin}, $Z^\lambda$ almost surely never returns to the origin after time $0$, and so pathwise uniqueness holds after time $0$, by \cite[Theorem 3.2 (ii)]{cherny_strong_2000}. On the other hand, suppose that $\lambda \in (\frac{\sqrt 2}{2}, 1)$. Then \cite[Theorem 3.2 (iii)]{cherny_strong_2000} shows that even uniqueness in law does not hold, and by \Cref{prop:lambda-return-origin} $Z^\lambda$ returns to the origin almost surely in finite time. Inspecting the proof of \cite[Theorem 3.2]{cherny_strong_2000}, we see that pathwise uniqueness for \eqref{eq:R-lambda-sde} holds after time $0$ up to the first hitting time of the origin. We will therefore only consider the SDE \eqref{eq:R-lambda-sde} up to the hitting time $\tau^\lambda_0$, as defined in \Cref{prop:lambda-return-origin}, in this case.
	
	Fix $\lambda \in (0, 1)$ and consider the angle process $\theta$, where we now omit the index $\lambda$ from our notation. By It\^o's formula, we calculate that $\theta$ satisfies
	\begin{equation}\label{eq:theta-lambda-sde}
		\D \theta_t = \sqrt{1 - \lambda^2} R^{-1}_t \D W_t - \lambda\sqrt{1 - \lambda^2} R^{-2}_t \D t, \quad t \in (0, \tau_0).
	\end{equation}
	Hence, given the value of $\theta_{\tau_\rho}$, for some $\rho > 0$ with $\tau_\rho < \tau_0$, the path of $(\theta_t)_{t \in (0, \tau_0)}$ is uniquely determined by \eqref{eq:theta-lambda-sde}. This observation will lead us to prove uniqueness in law for \eqref{eq:SDE-approx}.
\end{remark}	

We now turn to the proof of \Cref{thm:approx-no-strong}, where we show that there do not exist strong solutions of \eqref{eq:SDE-approx}, following a similar strategy to the proof of \Cref{thm:no-strong-solution}. Here, the angle process of a solution of \eqref{eq:SDE-approx} is no longer a circular Brownian motion, as was the case for solutions of \eqref{eq:SDE-gamma} in \Cref{prop:angle-cbm}. However, this process does have similar properties. We will show that, conditioned on the value of the radius, the angle process is uniformly distributed and independent of its increments. Here, we adapt \'Emery and Schachermayer's proof that the value of a circular Brownian motion at any time is uniformly distributed and independent of its increments, from Proposition 1 of \cite{azema_remark_1999}. We will deduce the result of \Cref{thm:approx-no-strong} from the following proposition.
	
\begin{proposition}\label[proposition]{prop:approx-unif-indep}
	Fix $\lambda \in (0, 1)$. For any weak solution $\left((X, W), \psp, \FF\right)$ of the SDE \eqref{eq:SDE-approx}, let $R$ be the radius process and $\theta$ the angle process, so that we can write
	\begin{equation}
		X_t = R_t
		\begin{bmatrix}
			\cos \theta_t\\
			\sin \theta_t
		\end{bmatrix}, \quad t > 0.
	\end{equation}
	Denote the hitting times of $R$ by
	\begin{equation}
		\tau_\rho \coloneqq  \inf\{t > 0 \colon R_t = \rho\}, \quad \rho \geq 0.
	\end{equation}
	Then, for any $\rho > 0$ with $\tau_\rho < \tau_0$,
	\begin{equation}
		\theta_{\tau_\rho} \sim \unif [0, 2\pi).
	\end{equation}
	Moreover, $\theta_{\tau_\rho}$ is independent of
	\begin{equation}
		\mathcal{H}_\infty \coloneqq  \sigma\left(\left\{\theta_t - \theta_s \colon 0 < s < t < \tau_0\right\}\right).
	\end{equation}
\end{proposition}

The above result relies in turn on the following technical lemma, which guarantees that the increments of the angle process at the hitting times of the radius process do not have a lattice distribution.
		
\begin{lemma}\label[lemma]{lem:coupling}
	Let $\theta$ be the angle process defined in \Cref{prop:approx-unif-indep} and fix $\rho > 0$ such that $\tau_\rho < \tau_0$. Then, for any $\phi \in [0, 2\pi)$,
	\begin{equation}\label{eq:non-arithmetic}
		\PP \left[\left(\theta_{\tau_{\rho}} - \theta_{\tau_{2^{-1}\rho}}\right) \in \left\{\phi + 2\pi m, \quad m \in \ZZ\right\}\right] < 1.
	\end{equation}
\end{lemma}

\begin{proof}
	Suppose for contradiction that there exists $\phi \in [0, 2 \pi)$ such that
	\begin{equation}\label{eq:arithmetic-distribution}
		\PP \left[\left(\theta_{\tau_{\rho}} - \theta_{\tau_{2^{-1}\rho}}\right) \in \left\{\phi + 2\pi m, \quad m \in \ZZ\right\}\right] = 1.
	\end{equation}
	Let $R$ be the radius process and $\theta$ the angle process, as defined in \Cref{prop:approx-unif-indep}, and recall that $R$ and $\theta$ satisfy the SDEs \eqref{eq:R-lambda-sde} and \eqref{eq:theta-lambda-sde}, respectively. 
	
	We will use a coupling argument to arrive at a contradiction. Consider two independent weak solutions $(R^1, \theta^1)$, $(R^2, \theta^2)$ of the SDEs \eqref{eq:R-lambda-sde} and \eqref{eq:theta-lambda-sde} on a common probability space. For $i = 1, 2$ and any $r \geq 0$,  denote the hitting time
	\begin{equation}
		\tau^i_r \coloneqq  \inf\{t > 0 \colon R^i_t = r\}.
	\end{equation}
	 Note that, as we observed in \Cref{rem:lambda-unique-strong}, given the value of $\theta$ at radius $2^{-1}\rho$, the process $\theta$ is uniquely defined via the SDE \eqref{eq:theta-lambda-sde} up to the first return to the origin.
	
	Fix $\psi^1, \psi^2 
	\in [0, 2\pi)$ such that $\psi^1 \not\equiv \psi^2 \mod 2\pi$, and shift $\theta^1$ and $\theta^2$ to define
	\begin{equation}
		\theta^{\psi^1}_t \coloneqq  \theta^1_t + \psi^1 -\theta^1_{\tau^1_{2^{-1}\rho}} \qandq \theta^{\psi^2}_t \coloneqq  \theta^2_t + \psi^2- \theta^2_{\tau^2_{2^{-1}\rho}}.
	\end{equation}
	Then, at the first hitting time of radius $2^{-1}\rho$, the values of the processes $\theta^{\psi^1}$ and $\theta^{\psi^2}$ are almost surely equal to $\psi^1$ and $\psi^2$, respectively, and these shifted processes still satisfy the SDE \eqref{eq:theta-lambda-sde}.
	
	Suppose that there exists some radius $\eta \in (2^{-1}\rho, \rho)$ such that
	\begin{equation}
		\theta^{
		\psi^1}_{\tau^1_\eta} = \theta^{\psi^2}_{\tau^2_\eta}.
	\end{equation}
	Then we can couple the two processes $\theta^{\psi^1}, \theta^{\psi^2}$ as follows. Define $\tilde{\theta}$ by
	\begin{equation}
		\tilde{\theta}_t =
		\begin{cases}
			\theta^{\psi^2}_t, & t < \tau^2_\eta,\\
			\theta^{\psi^1}_{\tau^1_\eta - \tau^2_\eta + t}, & t \geq \tau^2_\eta.
		\end{cases}
	\end{equation}
	Then we see that the trajectories of $(R^1, \theta^{\psi^1})$ and $(R^2, \tilde{\theta})$ coincide on the set $(\eta, R) \times [0, 2\pi)$. Moreover, by the Markov property, the process $\tilde{\theta}$ still satisfies the SDE \eqref{eq:theta-lambda-sde}. Therefore, by condition \eqref{eq:arithmetic-distribution},
	\begin{equation}
	\begin{split}
		\theta^{\psi^1}_{\tau_{\rho}} & \equiv \psi^1 + \phi \mod 2\pi,\\
		\tilde{\theta}_{\tau_{\rho}} & \equiv \psi^2 + \phi \mod 2\pi.
	\end{split}
	\end{equation}
	But, by our choice of $\psi^1, \psi^2$, the above values are not equal, contradicting the coupling of the trajectories. This shows that, on the set $(2^{-1}\rho, \rho) \times [0, 2\pi)$, the supports of $(R^1, \theta^{\psi^1})$ and $(R^2, \theta^{\psi^2})$ must be disjoint.
	
	Since our choice of the shifts $\psi^1$ and $\psi^2$ was arbitrary, the only feasible supports of $(R^i, \theta^{\psi^i})$ are the rays connecting the points $(2^{-1}\rho, \psi^i)$ and $(\rho, \psi^i)$, for $i = 1, 2$. This would imply that $\theta^{\psi^1}$ and $\theta^{\psi^2}$ are deterministic, but this is not the case for $\lambda < 1$.
	
	Hence there is no $\phi \in [0, 2\pi)$ such that \eqref{eq:arithmetic-distribution} holds.
\end{proof}

We now use this lemma to prove \Cref{prop:approx-unif-indep} on the uniformity and independence properties of the angle process.

\begin{proof}[Proof of \Cref{prop:approx-unif-indep}]
	Recall that $R$ satisfies the SDE \eqref{eq:R-lambda-sde} and $\theta$ satisfies the SDE \eqref{eq:theta-lambda-sde}.
	
	Fix $\rho > 0$ such that $\tau_\rho < \tau_0$. We show that $\theta_{\tau_\rho}$ is uniformly distributed on $[0, 2\pi)$ by using the characteristic function of the random variable $\theta_{\tau_\rho}$ on the torus, following the proof of Proposition 1 of \cite{azema_remark_1999}.
	For any $\phi \in \RR / 2\pi \ZZ$ and $k \in \ZZ$, define the characteristic function
	\begin{equation}
		e_k(\phi) \coloneqq  \exp\{iky\}, \quad \text{for any} \quad y \in \RR \quad \text{such that} \quad y \equiv \phi \mod 2\pi.
	\end{equation}
	Fix $k \in \ZZ \setminus \{0\}$ and $\rho_1 > 0$ with $\tau_\rho < \tau_0$. We aim to show that $\EE[e_k(\theta_{\tau_{\rho_1}})] = 0$.
	
	Let $\rho_0 \in (0, \rho_1)$. Then, writing
	\begin{equation}
		\theta_{\tau_{\rho_1}} = \theta_{\tau_{\rho_0}} + \int_{\tau_{\rho_0}}^{\tau_{\rho_1}}\D \theta_s,
	\end{equation}
	and denoting by $\EE$ the expectation with respect to the given probability measure $\PP$, we have
	\begin{equation}
	\begin{split}
		\abs{\EE \left[e_k(\theta_{\tau_{\rho_1}})\right]} & = \abs{\EE\left[e_k(\theta_{\tau_{\rho_0}})e_k(\theta_{\tau_{\rho_1}} - \theta_{\tau_{\rho_0}})\right]}.
	\end{split}
	\end{equation}
	In order to break up the expectation on the right hand side into the product of expectations, we use the following conditional independence. We see that future increments of $\theta$ depend only on the history of $\theta$ through the current value of $R$, since $R$ is Markovian. That is, for any $s < u < v$,
	\begin{equation}
		\theta_v - \theta_u \quad \text{conditioned on} \quad \sigma(R_s) \quad \text{is independent of} \quad \mathcal{F}^\theta_s.
	\end{equation}
	Now note that, taking $s = \tau_{\rho_0}$, the $\sigma$-algebra $\sigma(R_{\tau_{\rho_o}})$ is trivial, 	and so future increments of $\theta$ are independent of $\mathcal{F}^\theta_{\tau_{\rho_1}}$, without any conditioning. Hence
	\begin{equation}\label{eq:char-fn-indep}
		\EE\left[e_k(\theta_{\tau_{\rho_0}})e_k(\theta_{\tau_{\rho_1}} - \theta_{\tau_{\rho_0}})\right] = \EE\left[e_k(\theta_{\tau_{\rho_0}})\right]\EE\left[e_k(\theta_{\tau_{\rho_1}} - \theta_{\tau_{\rho_0}})\right].
	\end{equation}
			
	We will now consider the increment $\theta_{\tau_{\rho_1}} - \theta_{\tau_{\rho_0}}$. We claim that, for small radii $\rho_0$, the value of this increment approaches a uniform distribution on $[0, 2\pi)$. We show this by using a scaling argument, as follows.
	
	Fix $\alpha > 0$ and rescale time by defining $s \coloneqq  \alpha t$ for $t \geq 0$. Then, for $t \geq 0$, define
	\begin{equation}
		\tilde{B}^\alpha_s \coloneqq  \alpha^\frac{1}{2} B_t, \quad \tilde{R}^\alpha_s \coloneqq  \alpha^\frac{1}{2} R_t, \quad \tilde{\theta}^\alpha_s \coloneqq  \theta_t,
	\end{equation}
	so that
	\begin{equation}
		\D s = \alpha \D t, \quad \text{and} \quad \D \tilde{B}^\alpha_s = \alpha^\frac{1}{2}\D B_t.
	\end{equation}
	We can calculate
	\begin{equation}\label{eq:R-rescaled}
	\begin{split}
		\D \tilde{R}^\alpha_s & = \alpha^\frac{1}{2}\left(\lambda \D B_t + \frac{1 - \lambda^2}{2 R_t} \D t\right)\\
		& = \alpha^\frac{1}{2} \left(\lambda \alpha^{- \frac{1}{2}}\D \tilde{B}^\alpha_s + \frac{1 - \lambda^2}{2\alpha^{-\frac{1}{2}}\tilde{R}^\alpha_s}\alpha^{-1}\D s\right)\\
		& = \lambda \D \tilde{B}^\alpha_s + \frac{1 - \lambda^2}{2\tilde{R}^\alpha_s}\D s,
	\end{split}
	\end{equation}
	and
	\begin{equation}\label{eq:theta-rescaled}
	\begin{split}
		\D \tilde{\theta}^\alpha_s & = \sqrt{1 - \lambda^2}R_t^{-1}\D B_t - \lambda \sqrt{1 - \lambda^2}R_t^{-2} \D t\\
		& = \sqrt{1 - \lambda^2} \left(\alpha^{-\frac{1}{2}}\tilde{R}^\alpha_s\right)^{-1}\alpha^{-\frac{1}{2}} \D \tilde{B}^\alpha_s - \lambda \sqrt{1 - \lambda^2}\left(\alpha^{-\frac{1}{2}}\tilde{R}^\alpha_s\right)^{-2}\alpha^{-1} \D s\\
		& = \sqrt{1 - \lambda^2}\left(\tilde{R}^\alpha_s\right)^{-1} \D \tilde{B}^\alpha_s - \lambda\sqrt{1 - \lambda^2} \left(\tilde{R}^\alpha_s\right)^{-2} \D s.
	\end{split}
	\end{equation}
	And so, after this rescaling, $(\tilde{R}^\alpha, \tilde{B}^\alpha)$ and $(\tilde{\theta}^\alpha, \tilde{B}^\alpha)$ satisfy the same SDEs \eqref{eq:R-lambda-sde} and \eqref{eq:theta-lambda-sde} as $(R, B)$ and $(\theta, B)$.
	
	For $i = 0, 1$, let $\tilde{\tau}^0_{\rho_i}$ be the first time that the process $\tilde{R}^\alpha_s$ hits $\rho_i$, having started from the origin. Then we have the following equality in distribution:
	\begin{equation}
		\theta_{\tau_{\rho_1}} - \theta_{\tau_{\rho_0}} = \tilde{\theta}^\alpha_{\tilde{\tau}^0_{\sqrt{\alpha}\rho_1}} - \tilde{\theta}^\alpha_{\tilde{\tau}^0_{\sqrt{\alpha}\rho_0}} = \theta_{\tau_{\sqrt{\alpha}\rho_1}} - \theta_{\tau_{\sqrt{\alpha}\rho_0}},
	\end{equation}
	where the first equality holds pointwise by rescaling, and the second equality holds in distribution because the rescaled processes satisfy the same SDEs as the original processes.
	
	Moreover, recalling our observation that increments of $\theta$ between hitting times of $R$ are independent, we see that the increments 
	\begin{equation}
		\theta_{\tau_{\rho_1}} - \theta_{\tau_{\rho_0}} \qandq \theta_{\tau_{\sqrt{\alpha}\rho_1}} - \theta_{\tau_{\sqrt{\alpha}\rho_0}}
	\end{equation}
	are independent and identically distributed when $\sqrt{\alpha} \rho_1 \leq \rho_0$.
	
	Now let $N \in \NN$ and set $\rho_0 = 2^{-N} \rho_1$. We can write the increment of $\theta$ as a sum of i.i.d. random variables
	\begin{equation}
		\theta_{\tau_{\rho_1}} - \theta_{\tau_{\rho_0}} = \sum_{k = 0}^{N - 1}\left(\theta_{\tau_{2^{-k}\rho_1}} - \theta_{\tau_{2^{-k + 1}\rho_1}}\right),
	\end{equation}
	and so
	\begin{equation}
		\abs{\EE\left[e_k\left(\theta_{\tau_{\rho_1}} - \theta_{\tau_{\rho_0}}\right)\right]} = \abs{\EE\left[e_k\left(\theta_{\tau_{\rho_1}} - \theta_{\tau_{2^{-1}\rho_1}}\right)\right]}^N.
	\end{equation}
	By Jensen's inequality, 
	\begin{equation}\label{eq:jensen-char-fn}
		\abs{\EE\left[e_k\left(\theta_{\tau_{\rho_1}} - \theta_{\tau_{2^{-1}\rho_1}}\right)\right]}^2 \leq 1,
	\end{equation}
	with equality if and only if there exists $\phi \in [0, 2\pi)$ such that
	\begin{equation}
		\PP \left[\left(\theta_{\tau_{\rho_1}} - \theta_{\tau_{2^{-1}\rho_{1}}}\right) \in \left\{\phi + 2\pi m, \quad m \in \ZZ\right\}\right] = 1.
	\end{equation}
	By \Cref{lem:coupling}, no such $\phi$ exists, and so the inequality \eqref{eq:jensen-char-fn} is strict. We then have that
	\begin{equation}
		\abs{\EE\left[e_k\left(\theta_{\tau_{\rho_1}} - \theta_{\tau_{\rho_0}}\right)\right]} = \abs{\EE\left[e_k\left(\theta_{\tau_{\rho_1}} - \theta_{\tau_{2^{-1}\rho_1}}\right)\right]}^N \xrightarrow{N \to \infty} 0,
	\end{equation}

	Returning to our calculation of the characteristic function of $\theta_t$ in \eqref{eq:char-fn-indep}, we have
	\begin{equation}
	\begin{split}
		\abs{\EE \left[e_k(\theta_{\tau_{\rho_1}})\right]} & = \abs{\EE\left[e_k(\theta_{\tau_{\rho_0}})\right]\EE\left[e_k(\theta_{\tau_{\rho_1}} - \theta_{\tau_{\rho_0}})\right]}\\
		& \leq \abs{\EE\left[e_k(\theta_{\tau_{\rho_1}} - \theta_{\tau_{\rho_0}})\right]}\\
		& \xrightarrow{\rho_0 \downarrow 0} 0.
	\end{split}
	\end{equation}
	Hence $\theta_{\tau_{\rho_1}}$ is uniformly distributed on $[0, 2\pi)$.
	
	We now show that $\theta_{\tau_{\rho_1}}$ is independent of $\mathcal{H}_\infty$, the sigma-algebra generated by all increments of $\theta$.
	
	Let $\left(\rho_n\right)_{n \in \NN}$ be a strictly positive decreasing sequence with $\lim_{n \to \infty}\rho_n = 0$. For each $n \in \NN$, define
	\begin{equation}
		\mathcal{H}^n \coloneqq  \sigma\left(\left\{\theta_v - \theta_u \colon \tau_{\rho_n} \leq u \leq v\right\}\right),
	\end{equation}
	the sigma-algebra generated by all increments of $\theta$ after the first hitting time of $\rho_n$.
	
	Recalling that we are working with filtrations that satisfy the usual conditions, we have that $\mathcal{H}_\infty = \bigvee_{n \in \NN} \mathcal{H}^n$, since $\tau_{\rho_n} \to 0$ almost surely as $n \to \infty$. Therefore, by martingale convergence (see e.g.\ Theorem 4.3 of \cite[Chapter VII]{shiryaev_probability_1996}),
	\begin{equation}
		\EE\left[e_k\left(\theta_{\tau_{\rho_1}}\right) \vline \mathcal{H}^n\right] \xrightarrow{n \to \infty} \EE\left[e_k\left(\theta_{\tau_{\rho_1}}\right) \vline \mathcal{H}_\infty\right],
	\end{equation}
	in $\mathcal{L}^1$ and almost surely.
	
	We now fix $n \in \NN$ and consider
	\begin{equation}
	\begin{split}
		\abs{\EE\left[e_k\left(\theta_{\tau_{\rho_1}}\right) \vline \mathcal{H}^n\right]} & = \abs{\EE\left[e_k\left(\theta_{\tau_{\rho_n}}\right)e_k\left(\theta_{\tau_{\rho_1}} - \theta_{\tau_{\rho_n}}\right) \vline \mathcal{H}^n\right]}.
	\end{split}
	\end{equation}
	By the same conditional independence arguments as we used in the proof of uniformity, $\theta_{\tau_{\rho_n}}$ is independent of $\mathcal{H}^n$. Since $\tau_{\rho_1} \geq \tau_{\rho_n}$ pointwise, $\theta_{\tau_{\rho_1}} - \theta_{\tau_{\rho_n}}$ is $\mathcal{H}^n$-measurable. Therefore
	\begin{equation}
	\begin{split}
		\abs{\EE\left[e_k\left(\theta_{\tau_{\rho_n}}\right)e_k\left(\theta_{\tau_{\rho_1}} - \theta_{\tau_{\rho_n}}\right) \vline \mathcal{H}^n\right]} & = \abs{e_k\left(\theta_{\tau_{\rho_1}} - \theta_{\tau_{\rho_n}}\right)}\abs{\EE\left[e_k\left(\theta_{\tau_{\rho_n}}\right)\right]}\\
		& = \abs{\EE\left[e_k\left(\theta_{\tau_{\rho_n}}\right)\right]}\\
		& = 0,
	\end{split}
	\end{equation}
	by the uniformity of $\theta_{\tau_{\rho_n}}$.
	
	Hence
	\begin{equation}
		\abs{\EE\left[e_k\left(\theta_{\tau_{\rho_1}}\right) \vline \mathcal{H}^n\right]} = 0, \quad \text{for all} \quad n \in \NN,
	\end{equation}
	and so, by martingale convergence,
	\begin{equation}
		\EE\left[e_k\left(\theta_{\tau_{\rho_1}}\right) \vline \mathcal{H}_\infty\right] = 0.
	\end{equation}
	Taking $Y$ to be any bounded $\mathcal{H}_\infty$-measurable random variable, we then have
	\begin{equation}
		\EE\left[Y e_k\left(\theta_{\tau_{\rho_1}}\right)\right] = \EE\left[Y \EE\left[e_k\left(\theta_{\tau_{\rho_1}}\right) \vline \mathcal{H}_\infty\right]\right] = 0.
	\end{equation}		
	Hence $\theta_{\tau_{\rho_1}}$ is independent of $\mathcal{H}_\infty$.
\end{proof}

The following uniqueness result is an immediate corollary of \Cref{prop:approx-unif-indep}.

\begin{corollary}\label[corollary]{cor:approx-unique}
	Uniqueness in law holds for \eqref{eq:SDE-approx} up to the first hitting time of the origin.
\end{corollary}

\begin{proof}
	Given a pair of processes $(X, B)$ satisfying \eqref{eq:SDE-approx}, write $R$ and $\theta$ for the radius and angle processes of $X$, respectively. Then, as shown in \Cref{rem:lambda-unique-strong}, $R$ is the unique non-negative solution of \eqref{eq:R-lambda-sde}. Recall the notation $\tau_\rho = \inf\{t > 0 : \; R_t = \rho\}$ for $\rho \geq 0$. Then, by \Cref{rem:lambda-unique-strong} again, \eqref{eq:theta-lambda-sde} uniquely determines the path of $(\theta_t)_{t \in (0, \tau_0)}$, given the value of $\theta_{\tau_\rho}$ for some $\rho > 0$ with $\tau_\rho < \tau_0$. Moreover, we have $\theta_{\tau_\rho} \sim \unif[0, 2\pi)$, by \Cref{prop:approx-unif-indep}, and so the law of $(\theta_t)_{t \in (0, \tau_0)}$ is unique. Uniqueness in law for \eqref{eq:SDE-approx} on $[0, \tau_0)$ now follows.
\end{proof}

We now apply the independence result of \Cref{prop:approx-unif-indep} to conclude that the SDE \eqref{eq:SDE-approx} has no strong solution.

\begin{proof}[Proof of \Cref{thm:approx-no-strong}]The statement on uniqueness in law is proved in \Cref{cor:approx-unique}. Now suppose that $X$ is a strong solution of the SDE \eqref{eq:SDE-approx}. Then there is an $\RR_+$-valued $\FF^B$-adapted process $R$ satisfying the SDE \eqref{eq:R-lambda-sde} with $R_0 = 0$, and an $\RR/2\pi\ZZ$-valued $\FF^B$-adapted process $\theta$ satisfying the SDE \eqref{eq:theta-lambda-sde} such that
	\begin{equation}
		X_t = R_t
		\begin{bmatrix}
			\cos \theta_t\\
			\sin \theta_t
		\end{bmatrix},
		\quad t > 0.
	\end{equation}
	Recall the definition
	\begin{equation}
		\tau_{\rho} \coloneqq  \inf\left\{t > 0 \colon R_t = \rho\right\}, \quad \rho \geq 0,
	\end{equation}
	and fix $\rho > 0$ such that $\tau_\rho < \tau_0$.
	Then, by \Cref{prop:approx-unif-indep}, $\theta_{\tau_\rho}$ is independent of $\mathcal{H}_\infty$.
	
	Under our assumption that $\theta$ is adapted to $\FF^B$, this implies that
	\begin{equation}\label{eq:inclusion-hitting-time}
		\mathcal{H}_{\tau_\rho} \subsetneq \mathcal{F}^\theta_{\tau_\rho} \subseteq \mathcal{F}^B_{\tau_\rho}.
	\end{equation}
	However, we claim that $B$ is adapted to $\mathcal{H}$.
	
	To prove this claim, observe that, for any $0 < s < t < \tau_0$, the random variable
	\begin{equation}
		\langle \theta \rangle_t - \langle \theta \rangle_s = \int_s^t R_r^{-2} \D r
	\end{equation}
	is $\mathcal{H}_t$-measurable. Since $R_r > 0$ almost surely for $r > 0$, as we proved in \Cref{prop:lambda-return-origin}, $R_t$ is also $\mathcal{H}_t$-measurable.
	
	Now, from the SDE \eqref{eq:R-lambda-sde}, we have that
	\begin{equation}
		R_t - R_s = \lambda (B_t - B_s) + \int_s^t\frac{1 - \lambda^2}{2 R_r}\D r,
	\end{equation}
	and so $B_t - B_s$ is $\mathcal{F}^R_t$-measurable. Since $B_s \to 0$ as $s \to 0$, we can conclude that
	\begin{equation}\label{eq:inclusion-B-R-H}
		\mathcal{F}^B_t \subseteq \mathcal{F}^R_t \subseteq \mathcal {H}_t.
	\end{equation}
	Setting $t = \tau_{\rho}$ and combining the two inclusions  \eqref{eq:inclusion-hitting-time} and \eqref{eq:inclusion-B-R-H}, we arrive at the following contradiction:
	\begin{equation}
		\mathcal{F}^B_{\tau_{\rho}} \subseteq \mathcal{F}^R_{\tau_{\rho}} \subseteq \mathcal{H}_{\tau_{\rho}} \subsetneq \mathcal{F}^\theta_{\tau_{\rho}} \subseteq \mathcal{F}^B_{\tau_{\rho}}.
	\end{equation}
	Hence there is no strong solution of the SDE \eqref{eq:SDE-approx}.
	\end{proof}

\section{Application to a problem of stochastic control of martingales}\label{sec:control}

We now apply the result of \Cref{thm:no-strong-solution} to the control problem studied in \cite{cox_optimal_2021}. In \cite{cox_optimal_2021}, we find the value function for a $d$-dimensional control problem with radially symmetric running cost $f(x) = \tilde{f}(\lvert x \rvert)$, under mild regularity assumptions, including that $f$ is continuous away from the origin. We reformulate the control problem here as follows.

Fix $d \geq 2$, $R > 0$, and define the domain $D = \{x \in \RR^2 : \; |x| < R\}$. Also define the set of matrices $U \coloneqq  \left\{\sigma \in \RR^{d, d} \; \colon \trace(\sigma \sigma^\top) = 1\right\}$. The strong control problem is to find the value function $v^S: D \to \RR$, given by
\begin{equation}
	v^S(x) = \inf_{\nu \in \mathcal{U}}\EE^x \left[\int_0^\tau f(X^\nu_t) \D t\right], \quad x \in D,
\end{equation}
where $\mathcal{U}$ is the set of progressively measurable $U$-valued process, and for $\nu \in \mathcal{U}$, $X^\nu_t = x + \int_0^t \nu_s \D B_s$ and $\tau \coloneqq  \inf\{t \geq 0 \; \colon \; \lvert X^\nu_t \rvert \geq R\}$.
There is a corresponding weak version of the control problem, to find the weak value function $v^W$, where we optimise over solutions of martingale problems, rather than stochastic integrals.

 In dimension $d = 2$, the result \cite[Theorem 5.12]{cox_optimal_2021} does not treat the strong control problem at the origin under the condition that $\int_0^r \tilde{f}(s) \D s = \infty$ but $\int_0^r s\tilde{f}(s)\D s < \infty$ for all $r > 0$. We will apply \Cref{thm:no-strong-solution} to extend the result to this case.
 
 We recall the definition of the candidate value function $V: D \to \RR$ from \cite[Definition 4.6]{cox_optimal_2021}.
	For $k \in \NN$ and $i = 0, \dotso, k$, introduce the constant
	\begin{equation}
		\mathfrak{F}^k_i \coloneqq  2\sum_{j = i + 1}^k \left[(r_j - s_{j - 1}) s_{j - 1} \tilde{f}(s_{j - 1}) + \int_{s_{j - 1}}^{r_j} \int_{s_{j -1}}^s \tilde{f}(t) \D t \D s + \int_{r_j}^{s_j} s \tilde{f}(s) \D s\right],
	\end{equation}
	and consider the following two cases.
	If $\tilde{f}$ is increasing in $(0, \eta)$, then set $s_0 = 0$, define $(s_i, r_i)$ by 
	\begin{equation}\label{eq:caseI-switching}
		\begin{split}
			s_{i} & \coloneqq \inf\left\{r > r_i \colon \tilde{f}^\prime_+(s) > 0\right\},\quad
			r_{i+1} \coloneqq  \inf\left\{r > s_i \colon s_i \tilde{f}(s_i) + \int_{s_i}^r \tilde{f}(s) \D s > r \tilde{f}(r)\right\},
		\end{split}
	\end{equation} and let $K \in \NN$ be such that $R \in (s_{K - 1}, s_K]$. For $x \in D$, define
	\begin{equation}
	\begin{split}
		V(x) & = - 2 \int_{R \vee r_K}^{s_K} s \tilde{f}(s) \D s - 2(r_K - R \wedge r_K) s_{K - 1} \tilde{f}(s_{K - 1}) - 2 \int_{R \wedge r_K}^{r_K} \int_{s_{K - 1}}^s \tilde{f}(t) \D t \D s\\
		& \quad + 2 \sum_{i = 1}^K \ind{(s_{i - 1}, s_i]}(\abs{x}) \left[(r_i - \abs{x} \wedge r_i) s_{i - 1} \tilde{f}(s_{i -1 }) + \int_{\abs{x} \wedge r_i}^{r_i} \int_{s_{i - 1}}^s \tilde{f}(t) \D t \D s + \int_{\abs{x} \vee r_i}^{s_i} s \tilde{f}(s) \D s + \mathfrak{F}^K_i\right].
	\end{split}
	\end{equation}
	If $\tilde{f}$ is decreasing in $(0, \eta)$, then set $r_0 = 0$, define $(r_i, s_i)$ by
	\begin{equation}\label{eq:caseII-switching}
		\begin{split}
			r_{i+1} & \coloneqq  \inf\left\{r > s_i \colon s_i \tilde{f}(s_i) + \int_{s_i}^r \tilde{f}(s) \D s > r \tilde{f}(r)\right\}, \quad s_{i+1} \coloneqq \inf\left\{r > r_{i + 1} \colon \tilde{f}^\prime_+(r) > 0\right\},
		\end{split}
	\end{equation}
	and let $L \in \NN$ be such that $R \in (r_L, r_{L + 1}]$. For $x \in D$, define
	\begin{equation}
	\begin{split}
		V(x) & = - 2 \int_{R \wedge s_L}^{s_L} s \tilde{f}(s) \D s + 2(R \vee s_L - s_L) s_L \tilde{f}(s_L) + 2 \int_{s_L}^{R \vee s_L} \int_{s_L}^s \tilde{f}(t) \D t \D s\\
		& \quad + 2 \sum_{i = 0}^L \ind{(r_i, r_{i + 1}]}(\abs{x}) \left[\int_{\abs{x} \wedge s_i}^{s_i} s \tilde{f}(s) \D s  - (\abs{x} \vee s_i - s_i) s_i \tilde{f}(s_i) - \int_{s_i}^{\abs{x} \vee s_i} \int_{s_i}^s \tilde{f}(t) \D t \D s + \mathfrak{F}^L_i \right].
	\end{split}
	\end{equation}
Note in particular that, in the case that $\tilde{f}$ is decreasing on the interval $(0, \eta)$, $V$ satisfies
\begin{equation}\label{eq:value-origin-dpp}
	V(0) = 2 \int_0^\eta \xi \tilde{f}(\xi) \D \xi + V(y),
\end{equation}
for any $y \in D$ with $\lvert y \rvert = \eta$.

The generalisation of \cite[Theorem 5.12]{cox_optimal_2021} is the following.

\begin{theorem}\label{thm:strong-extension}
	Suppose that $\tilde{f}$ is continuous on $(0, R)$ and monotone on some interval $(0, \eta) \subset (0, R)$, and that the one-sided derivative $\tilde{f}^\prime_+(r)$ exists for all $r \in (0, R)$ and changes sign only finitely many times. Then
	\vspace{-1ex}
	\begin{equation}
		v^S(x) = v^W(x) =
		\begin{cases}
			- \infty, & \text{if} \; \int_0^r \tilde{f}(s) \D s = - \infty, \; \text{for any} \; r > 0,\\
			+\infty, & \text{if} \; x = 0 \; \text{and} \; \int_0^r s \tilde{f}(s) \D s = \infty, \; \text{for any} \; r > 0,\\
			V(x) \in (- \infty, +\infty), & \text{otherwise}.
		\end{cases}
	\end{equation}
\end{theorem}

In \cite{cox_optimal_2021} we show that the weak value function $v^W$ has the form given in \Cref{thm:strong-extension}. Moreover, we show that $v^S(x) = v^W(x)$ except possibly in the two-dimensional case at the origin when
\begin{equation}\label{eq:decr-inter-growth}
	\int_0^r \tilde{f}(s) \D s = \infty, \quad \text{but} \quad \int_0^r s\tilde{f}(s)\D s < \infty, \quad \text{for all} \quad r > 0.
\end{equation}
To conclude the proof of \Cref{thm:strong-extension} we will now show that $v^S(0) = V(0)$ under the above conditions.

\subsection{Equivalence of weak and strong control problems}\label{sec:weak-equals-strong}

Fix a probability space $(\tilde{\Omega}, \tilde{\mathcal{F}}, \tilde{\PP})$ on which a $\RR$-valued Brownian motion $B$ is defined with natural filtration $\FF^B = (F^B_t)_{t \geq 0}$. We know that there exists a weak solution $\left((X, W), \psp, \FF\right)$ of \eqref{eq:SDE-gamma} by Theorem 4.3 of \cite{larsson_relative_2021}. For $\tau_\eta\coloneqq  \inf \{t \geq 0 \; \colon \; \lvert X_t \rvert \geq \eta\}$, we can calculate that
	\begin{equation}
		\EE^0 \left[\int_0^{\tau_\eta} f( X_s)\D s\right] = 2 \int_0^\eta \xi \tilde{f}(\xi) \D \xi.
	\end{equation}
	We will show that there exists an $\FF^B$-martingale $\tilde{X}$ that is equal in law to $X$. This is the key step required to complete the proof of \Cref{thm:strong-extension}.

We will make use of the notion of \emph{isomorphisms} between filtered probability spaces in the following proof. We take the following definitions from the paper \cite{azema_vershiks_2001} of \'Emery and Schachermayer.	

\begin{definition}[Isomorphism]\label[definition]{def:isomorphism}
	Given a probability space $\psp$, denote the set of random variables on that probability space by $\mathcal{L}^0\psp$. An \emph{embedding} of $\psp$ into another probability space $(\overline{\Omega}, \overline{\mathcal{F}}, \overline{\PP})$ is a map
	\begin{equation}
		\Psi: \mathcal{L}^0\psp \to \mathcal{L}^0(\overline{\Omega}, \overline{\mathcal{F}}, \overline{\PP})
	\end{equation}
	that commutes with Borel operations on random variables and preserves probability laws.
	
	An \emph{isomorphism} from $\psp$ to $(\overline{\Omega}, \overline{\mathcal{F}}, \overline{\PP})$ is an embedding that is bijective.
\end{definition}

\begin{remark}
	We follow the same convention as in \cite{donati-martin_standardness_2011}  and also write $\Psi$ for the map in the above definition acting on sigma-algebras, stochastic processes and filtrations.
\end{remark}

\begin{definition}\label[definition]{def:filter-isomorphism}
	Two filtered probability spaces $(\Omega, \mathcal{F}, \PP, \FF)$ and $(\overline{\Omega}, \overline{\mathcal{F}}, \overline{\PP}, \overline{\FF})$, with $\FF = (\mathcal{F}_t)_{t \geq 0}$ and $\overline{\FF} = (\overline{\mathcal{F}}_t)_{t \geq 0}$, are \emph{isomorphic} if there exists an isomorphism
	\begin{equation}
		\Psi: \mathcal{L}^0(\Omega, \mathcal{F}_\infty, \PP) \to \mathcal{L}^0(\overline{\Omega}, \overline{\mathcal{F}}_\infty, \overline{\PP})
	\end{equation}
	such that $\Psi(\FF) = \overline{\FF}$.
\end{definition}

In \cite{donati-martin_standardness_2011}, Laurent gives similar definitions to the above for filtrations in discrete negative time. We will refer to results from \cite{donati-martin_standardness_2011} in the following proof.

\begin{proof}[Proof of \Cref{thm:strong-extension}]
	It is shown in \cite[Theorem 5.12]{cox_optimal_2021} that the conclusion of the theorem holds in all cases except in dimension $d = 2$ at the origin, under the conditions
	\begin{equation}
		\int_0^r \tilde{f}(s) \D s = \infty, \qandq \int_0^r s \tilde{f}(s) \D s < \infty , \quad \text{for any} \quad r > 0.
	\end{equation} In this case, \cite[Lemma 5.11]{cox_optimal_2021} shows that $v^W(0) = V(0)$. We now prove that $v^S(0) = V(0)$, thus completing the proof of the theorem.
	
	Fix a probability space $(\tilde{\Omega}, \tilde{\mathcal{F}}, \tilde{\PP})$ on which a $\RR$-valued Brownian motion $B$ is defined with natural filtration $\FF^B = (\mathcal F^B_t)_{t \geq 0}$, and recall the definition of the control set $\mathcal{U}$ given above. We will construct an $\FF^B$-martingale $X^{\tilde{\nu}}$ such that, for any $t \geq 0$,
	\begin{equation}
		X^{\tilde{\nu}}_t = \int_0^t \tilde{\nu}_s \D B_s,
	\end{equation}
	for some $\tilde{\nu} \in \mathcal{U}$, and
	\begin{equation}
		\EE^0\left[\int_0^{\tau_\eta} f(X^{\tilde{\nu}}_s) \D s \right] = 2 \int_0^\eta \xi \tilde{f}(\xi) \D \xi.
	\end{equation}
	
	By \Cref{thm:no-strong-solution}, there exists a unique (in law) weak solution $\left((X, B^\prime), \psp, \FF\right)$ of the SDE \eqref{eq:SDE-gamma} and that the process $X$ generates a Brownian filtration. That is, there exists a Brownian motion $W$ on the probability space $\psp$ with natural filtration $\FF^W = (\mathcal{F}^W_t)_{t \geq 0}$ such that the natural filtration of $X$ is equal to $\FF^W$.
	
	Since $B$ and $W$ are both $\RR$-valued Brownian motions, they have have the same law and so, as noted in Section 1.6 of \cite{donati-martin_standardness_2011}, the filtered probability spaces
	\begin{equation}
		(\tilde{\Omega}, \tilde{\mathcal{F}}, \tilde{\PP}, \FF^B) \qandq (\Omega, \mathcal{F}, \PP, \FF^W)
	\end{equation}
	are isomorphic, as defined in \Cref{def:filter-isomorphism}. That is, there exists an isomorphism
	\begin{equation}
		\Psi: \mathcal{L}^0(\Omega, \mathcal{F}^W_\infty, \PP) \to \mathcal{L}^0(\tilde{\Omega}, \tilde{\mathcal{F}}^B_\infty, \tilde{\PP}),
	\end{equation}
	as defined in \Cref{def:isomorphism}, such that
	\begin{equation}
		\Psi(\FF^W) = \FF^B.
	\end{equation}
	Now define a process $\tilde{X}$ on the probability space $(\tilde{\Omega}, \mathcal{F}^B_\infty, \tilde{\PP})$ by
	\begin{equation}
		(\tilde{X}_t)_{t \geq 0} = \Psi\left((X_t)_{t \geq 0}\right).
	\end{equation}
	For any $t \geq 0$, we have that $\Psi(\mathcal{F}^W_t) = \mathcal{F}^B_t$ and
	$
		\Psi: \mathcal{L}^0(\Omega, \mathcal{F}^W_t, \PP) \to \mathcal{L}^0(\tilde{\Omega}, \tilde{\mathcal{F}}^B_t, \tilde{\PP})
	$
	is an isomorphism, as noted after the definition of an isomorphism in \cite{azema_vershiks_2001}. Therefore, since $X$ is adapted to $\FF^W$, it follows that $\tilde{X}$ is adapted to $\FF^B$.
	
	Now fix $0 < s < t$. Then, using Lemma 5.3 of \cite{donati-martin_standardness_2011} to apply the isomorphism $\Psi$ to a conditional expectation, we see that
	\begin{equation}
	\begin{split}
		\EE^{\tilde{\PP}}\left[\tilde{X}_t \;\big\vert\; \mathcal{F}^B_s \right] & = \EE^{\tilde{\PP}}\left[\Psi(X_t) \;\big\vert\; \Psi(\mathcal{F}^W_s)\right]\\
		& = \Psi\left(\EE^\PP\left[X_t \;\vert\; \mathcal{F}^W_s\right]\right) = \Psi(X_s) = \tilde{X}_s,
	\end{split}
	\end{equation}
	where the third equality follows from the fact that $X$ is an $\FF^W$-martingale. Hence $\tilde{X}$ is an $\FF^B$-martingale. By the definition of an isomorphism in \Cref{def:isomorphism}, we also have that the processes $X$ and $\tilde{X}$ are equal in law.
	
	We now apply the martingale representation theorem, as found for example in Theorem 3.4 of \cite[Chapter 5]{revuz_continuous_1999}. This result implies that $t \mapsto \tilde{X}_t$ is continuous and there exists an $\FF^B$-progressively measurable $\RR$-valued process $\tilde{\nu}$ such that, for any $t \geq 0$, we have the representation
	\begin{equation}\label{eq:strong-martingale-rep}
		\tilde{X}_t = \int_0^t \tilde{\nu}_s \D B_s.
	\end{equation}
	
	We can also deduce that $\tilde{X}$ has quadratic variation $t \mapsto \langle \tilde{X} \rangle_t = t$, as follows. The quadratic variation of $X$ is $t \mapsto \langle X \rangle_t = t$, and so $t \mapsto \abs{X_t}^2 - t$ is an $\FF^W$-martingale. Using Lemma 5.3 of \cite{donati-martin_standardness_2011} again, we calculate that, for any $0 < s < t$,
	\begin{equation}
	\begin{split}
		\EE^{\tilde{\PP}}\left[|\tilde{X}_t|^2 - t \;\big\vert\; \mathcal{F}^B_s\right] & = \EE^{\tilde{\PP}}\left[\Psi(\abs{X_t}^2) \;\big\vert\; \Psi(\mathcal{F}^W_s)\right] - t\\
		& = \Psi\left(\EE^\PP\left[\abs{X_t}^2 \;\vert\; \mathcal{F}^W_s\right]\right) - t\\
		& = \Psi(\abs{X_s}^2 + t - s) - t\\
		& = |\tilde{X}_s|^2 - s.
	\end{split}
	\end{equation}
	Hence $t \mapsto |\tilde{X}_t|^2 - t$ is an $\FF^B$-martingale and so, for any $t \geq 0$, $\langle \tilde{X} \rangle_t = t$. From the representation \eqref{eq:strong-martingale-rep}, we also have that
	\begin{equation}
		t \mapsto \langle \tilde{X} \rangle_t = \int_0^t \trace(\tilde{\nu}_s \tilde{\nu}_s^\top) \D s.
	\end{equation}
	Hence $\trace(\tilde{\nu}_t \tilde{\nu}_t^\top) = 1$, for any $t \geq 0$, and so $\tilde{\nu} \in \mathcal{U}$.
	
	Since the process $\tilde{X}$ has the same law as $X$, we have
	\begin{equation}
		\EE^{\tilde\PP} \left[\int_0^{\tau_\eta} f(\tilde{X}_s)\D s\right] = \EE^{\PP} \left[\int_0^{\tau_\eta} f( X_s)\D s\right] = 2 \int_0^\eta \xi \tilde{f}(\xi) \D \xi.
	\end{equation}
	Therefore, for any $y \in D$ with $\abs{y} = \eta$, the dynamic programming principle given in \cite[Proposition A.2]{cox_optimal_2021} implies that
	\begin{equation}
	\begin{split}
		v^S(0) & \leq \EE^{\tilde{\PP}}\left[\int_0^{\tau_\eta}f(\tilde{X}_s) \D s + v^S(\tilde{X}_{\tau_\eta})\right]\\
		& = 2 \int_0^\eta \xi \tilde{f}(\xi) \D \xi + V(y)\\
		& = V(0),
	\end{split}
	\end{equation}
	where the final equality comes from \eqref{eq:value-origin-dpp}.

	Using the result that $v^W(0) = V(0)$ from Lemma 5.11 of \cite{cox_optimal_2021}, we have
	\begin{equation}
		v^S(0) \leq V(0) = v^W(0) \leq v^S(0),
	\end{equation}
	and conclude that
	\begin{equation}
		v^S(0) = v^W(0) = V(0).
	\end{equation}
\end{proof}

\subsection{Feedback controls}\label{sec:feedback}

A further question of interest is whether the value function remains the same when we restrict to \emph{feedback controls}. A control $\nu \in \mathcal{U}$ is a feedback control if it is of the form $\nu_t = \sigma(X^{\sigma, x}_t)$, where $X^{\sigma, x}$ is a strong solution of the SDE $\D X_t = \sigma(X_t) \D B_t$, $X_0 = x$, for some Borel function $\sigma: D \to U$. Such controls are also referred to as \emph{Markov controls} in the thesis \cite{robinson_stochastic_2020} and are defined similarly in Section 3 of \cite[Chapter IV]{fleming_controlled_2006} and Section 3.1 of \cite{touzi_optimal_2013}.

In our problem, if there were a strong solution of the SDE \eqref{eq:SDE-gamma}, then this would give an optimal feedback control. Having shown that this is not the case in \Cref{thm:no-strong-solution}, we are left with the open questions of whether the value function can be attained by some feedback control, and whether it can be approximated by a sequence of such controls.
In regards to the first question, we have the following result, whose proof can be found in the thesis \cite{robinson_stochastic_2020}.

\begin{proposition}\label[proposition]{prop:derivation-SDE}
	Let $\psp$ be a probability space on which an $\RR^2$-valued Brownian motion is defined with natural filtration $\FF^B$. Suppose there exists a Borel function $\sigma: D \to U$ such that the SDE
	\begin{equation}
		\D X_t = \sigma(X_t) \D B_t, \quad X_0 = 0,
	\end{equation}
	has a strong solution $X$ with $t \mapsto \abs{X_t}$ deterministically increasing. Then there exists a Borel function $\gamma: D \to \{x \in \RR^2 \colon \abs{x} = 1\}$ such that, for any $x = (x_1, x_2)^\top \in D$,
	\begin{equation}
		\sigma(x) = \frac{1}{\abs{x}}
		\begin{bmatrix}
			-x_2\\
			x_1
		\end{bmatrix}
		\gamma(x)^\top.
	\end{equation}
	Moreover, $\abs{X_t} = \sqrt{t}$, for all $t \geq 0$, and for $R > 0$, $\eta \in (0, R)$, and $\tilde{f}: [0, R)$ a continuous function, we have
	\begin{equation}
		\EE^0\left[\int_0^{\tau_\eta} \tilde{f}(\abs{X_s}) \D s\right] = 2 \int_0^\eta \xi \tilde{f}(\xi) \D \xi.
	\end{equation}
\end{proposition}
Therefore to answer the question of existence of feedback controls that attain the strong value function, one needs to study the existence of strong solutions of SDEs of the form given in \Cref{prop:derivation-SDE}. The methods used in \Cref{sec:non-existence} do not apply directly here when we now consider a two-dimensional Brownian motion.

The result of \Cref{thm:approx-no-strong} contributes to investigating whether the value function can be approximated by feedback controls, as the next result shows.

\begin{proposition}\label[proposition]{prop:lambda-approx-cost}
	Suppose that the growth condition \eqref{eq:decr-inter-growth} holds and let $\eta > 0$ be such that $\tilde{f}$ is decreasing and positive on the interval $(0, \eta)$. For $\lambda \in (0, \frac{\sqrt{2}}{2})$, if $X^\lambda$ satisfies the SDE \eqref{eq:SDE-approx}, then
	\begin{equation}
		0 \leq \EE^0\left[\int_0^{\tau_\eta}f(X^\lambda_s)\D s\right] < \infty,
	\end{equation}
	and, moreover,
	\begin{equation}
		\lim_{\lambda \downarrow 0}\EE^0\left[\int_0^{\tau_\eta}f(X^\lambda_s)\D s\right] = 2 \int_0^\eta \xi \tilde{f}(\xi) \D \xi.
	\end{equation}
\end{proposition}

Again, we refer to the thesis \cite{robinson_stochastic_2020} for a proof. Since \Cref{thm:approx-no-strong} shows that \eqref{eq:SDE-approx} has no strong solution for $\lambda \in (0, \frac{\sqrt{2}}{2})$, we have ruled out one potential approximating sequence of feedback controls. We leave open the question of whether the values coincide for the problems of optimising over strong controls and over feedback controls.

\bibliographystyle{abbrv}
\bibliography{bibliography.bib}

\end{document}